\documentclass[times]{oupau}
\usepackage{amsmath,amssymb,amsthm,xypic,times,pslatex,bbm,amscd} 
\input xy
\xyoption{all}

\usepackage{xspace}

\usepackage[pdfstartview=FitH,
            colorlinks,
           linkcolor=reference,
            citecolor=citation,
            urlcolor=e-mail,
           backref]{hyperref}

\RequirePackage{verbatim}
\RequirePackage{eufrak}
\RequirePackage{microtype}
\RequirePackage{hyperref}

\RequirePackage{color}
\definecolor{todo}{rgb}{1,0,0}
\definecolor{answer}{rgb}{0,0,1}
\definecolor{new}{rgb}{1,0,1}
\definecolor{conditional}{rgb}{0,1,0}
\definecolor{e-mail}{rgb}{0,.40,.80}
\definecolor{reference}{rgb}{.20,.60,.22}
\definecolor{mrnumber}{rgb}{.80,.40,0}
\definecolor{citation}{rgb}{0,.40,.80}

\DeclareMathAlphabet{\mathpzc}{OT1}{pzc}{m}{it}
\RequirePackage{dsfont}

\theoremstyle{plain}
\newtheorem{theo}{Theorem}[section]
\newtheorem{prop}[theo]{Proposition}
\newtheorem{cor}[theo]{Corollary}
\theoremstyle{definition}
\newtheorem{defi}[theo]{Definition}

\theoremstyle{remark}
\newtheorem{rem}[theo]{Remark}

\theoremstyle{definition}
\newtheorem{ex}[theo]{Example}

\newcommand{\PV}{PV}
\newcommand{\f}{\phi}

\newcommand{\ida}{\mathfrak{a}}
\newcommand{\idb}{\mathfrak{b}}

\newcommand{\p}{\mathfrak{p}}
\newcommand{\q}{\mathfrak{q}}
\newcommand{\m}{\mathfrak{m}}

\newcommand{\spec}{\operatorname{Spec}}

\newcommand{\CC}{\mathcal C}
\newcommand{\ZZ}{\mathbb{Z}}
\newcommand{\NN}{\mathbb{N}}
\newcommand{\QQ}{\mathbb{Q}}

\newcommand{\Gl}{\operatorname{GL}}
\newcommand{\Sl}{\operatorname{SL}}

\newcommand{\AutKfs}{\operatorname{Aut}_{\fs}}
\newcommand{\AlgKs}{\operatorname{Alg}_{k\text{-}\s}}
\newcommand{\Ks}{{k\text{-}\s}}
\newcommand{\Hom}{\operatorname{Hom}}

\newcommand{\Q}{\mathrm{Quot}}

\newcommand{\id}{\operatorname{id}}

\newcommand{\sdim}{\s\text{-}\dim}
\newcommand{\C}{\mathbb{C}}
\newcommand{\kk}{\mathbf{k}}
\newcommand{\Alg}{\mathrm{Alg}}
\newcommand{\Set}{\mathrm{Sets}}
\newcommand{\Gm}{\mathbf{G_m}}

\newcommand{\quot}{\operatorname{Quot}}
\newcommand{\Char}{\operatorname{char}}
\newcommand{\Gal}{\operatorname{Gal}}

\newcommand{\s}{\sigma}

\newcommand{\fs}{{\phi\sigma}}
\newcommand{\fpsf}{$\phi$-pseudo $\sigma$-field\xspace}
\newcommand{\per}{\operatorname{period}}
\newcommand{\seq}{\operatorname{Seq}}

\newcommand{\Ge}{\geqslant}

\begin{document}

\author{Alexey Ovchinnikov\affil{1} and Michael Wibmer\affil{2}}
\address{
\affilnum{1} CUNY Queens College, Department of Mathematics, 65-30 Kissena Blvd,
Queens, NY 11367-1597, USA\\
\affilnum{2} RWTH Aachen, 52056, Aachen, Germany}

\correspdetails{aovchinnikov@qc.cuny.edu\, and\, michael.wibmer@matha.rwth-aachen.de}

\title{$\sigma$-Galois theory of linear difference equations}

\begin{abstract} We develop a Galois theory for systems of linear difference equations with an action of an endomorphism $\s$. This provides a technique to
test whether solutions of such systems satisfy $\s$-polynomial equations and, if yes, then characterize those. We also show how to apply our work to study isomonodromic difference equations and difference algebraic properties of meromorphic functions.
\end{abstract}

\maketitle

\section{Introduction}
Inspired by the numerous applications of the differential algebraic independence results from~\cite{HardouinSinger:DifferentialGaloisTheoryofLinearDifferenceEquations}, we develop a Galois theory with an action of an endomorphism $\s$ for systems of linear difference equations of the form $\f (y)=Ay$, where $A \in \Gl_n(K)$ and $K$ is 
 a $\fs$-field, that is, a field with two given commuting endomorphisms $\f$ and $\s$, like in Example~\ref{ex:fsfields}. This provides a technique to
test whether solutions of such systems satisfy $\s$-polynomial equations and, if yes, then characterize those. 
Galois groups in this approach are groups of invertible matrices defined by $\s$-polynomial equations with coefficients in the $\s$-field  $K^\f := \{a \in K\:|\: \f(a) = a\}$. In more technical terms, such groups are functors from $K^\f$-$\s$-algebras to sets represented by finitely $\s$-generated $K^\f$-$\s$-Hopf algebras \cite{DiVizioHardouinWibmer:DifferenceGaloisofDifferential} . 
Also, our work is a highly non-trivial generalization of \cite{OvchinnikovTrushinetal:GaloisTheory_of_difference_equations_with_periodic_parameters}, where similar problems were considered but $\s$ was required to be of finite order (there exists $n$ such that $\s^n = \id$).

 Our main result is a construction of a $\s$-Picard--Vessiot ($\s$-PV) extension (see Theorem~\ref{theo: existence of PV ring}), that is, a minimal $\fs$-extension of the base $\fs$-field $K$ containing solutions of $\f(y)=Ay$. It turns out that the standard constructions and proofs in the previously existing difference Galois theories do not work in our setting. Indeed, this is mainly due to the reason that even if the field $K^\f$ is $\s$-closed \cite{Trushin:DifferenceNullstellensatz}, consistent  systems of $\s$-equations (such that the equation $1=0$ is not a $\s$-algebraic consequence of the system) with coefficients in $K^\f$ might not have a solution with coordinates in $K^\f$ (see more details in Remarks~\ref{rem:subtle} and~\ref{rem:subtle2}). However, our method avoids this issue. In our approach, a $\s$-PV extension is built iteratively (applying $\s$), by carefully choosing a suitable usual \PV\ extension \cite{SingerPut:difference} at each step, and then ``patching'' them together. This is a difficult problem and requires several preparatory steps as described in \S\ref{sec:existence}. A similar approach was also taken in \cite[Thm.~8]{Wibmer:existence} for systems of differential equations with parameters. However, our case is more subtle and, as a result, requires more work.

Galois theory of difference equations $\f(y) =Ay$ without the action of $\s$ was studied in \cite{SingerPut:difference,ChatzidakisHardouinSinger:OntheDefinitionsOfDifferenceGaloisgroups,Amano,AmanoMasuoka:artiniansimple,AmanoMasuokaTakeuchi:HopfPVtheory,Andre,Wibmer:OnTheGaloisTheoryOfStronglyNormal}, with a non-linear generalization considered in \cite{Granier,Morikawa}, as well as with an action of a derivation $\partial$ in~\cite{CharlotteArxiv,CharlotteComp,HardouinSinger:DifferentialGaloisTheoryofLinearDifferenceEquations,CharlotteLucia11,CharlotteLucia12,CharlotteLucia13,DiVizioHardouin:DescentandConfluence,HDV,LuciaSMF}. The latter works provide algebraic methods  to test whether solutions of difference equations satisfy polynomial differential equations (see also \cite{Moshe2012} for a general Tannakian approach). In particular,
these methods can be used to prove H\"older's theorem that states that the $\Gamma$-function, which satisfies the difference equation $\Gamma(x+1)=x\cdot\Gamma(x)$,
satisfies no non-trivial differential equation over $\C(x)$. A Galois theory of differential equations $\partial(y)=Ay$ (the matrix $A$ does not have to be invertible in this case) with an action of $\s$ was also developed in~\cite{DiVizioHardouinWibmer:DifferenceGaloisofDifferential}.

Our work has numerous applications to studying difference and differential algebraic properties of functions. 
Isomonodromic $q$-difference equations, which lead to $q$-difference Painlev\'e equations, have been recently studied in~\cite{Hay,Murata,Joshi1,Joshi2}. In Theorem~\ref{thm:isomcrit}, we show how this property can
be detected using our $\s$-PV theory. On the other hand,
Theorem~\ref{thm:app} gives a general $\s$-algebraic independence (called difference hypertranscendency in~\cite{Takano}) test for first-order $\f$-difference equations. Theorem~\ref{thm:Nev} translates this to a $\s$-algebraic dependence test over the field of meromorphic functions with Nevanlinna growth order less than $1$ (see~\eqref{eq:Nevlimit}). 
It turns out that our methods allow us to generalize a modification (Lemma~\ref{lem:fromM1}) of complex-analytic results from \cite{BankKaufman}, which is another interesting application.
Theorem~\ref{thm:6} combined with either Theorems~\ref{thm:app} or~\ref{thm:Nev} can be used as computational tool. We illustrate this in 
Examples~\ref{ex:Gmmaf}  
 and~\ref{ex:power} as well as show how our work could
be used to study differential algebraic properties of functions given by power series in Example~\ref{ex:differential}.
Not only do we show practical applications of our work, we also hope that our theory will be applied in the future in diverse areas, such as described in~\cite{SalvyISSAC,SalvyJSC} and the papers on isomonodromic $q$-difference equations mentioned above.

The paper is organized as follows. We start with the basic definitions, notation, and review of existing results in \S\S\ref{sec:basicdefs}, \ref{sec:classicalPV}. We then introduce $\s$-PV extensions and study their basic properties in \S\ref{sec:sPVdef}. The main result, existence of $\s$-\PV\ extensions,  is contained in \S\ref{sec:existence}, which starts by developing the needed technical tools. We extend the main result in \S\ref{subsec: special existence} to include more useful situations in which $\s$-\PV\ extensions exist. Uniqueness for $\s$-\PV\ extensions is established in \S\ref{sec:uniqueness}. We recall from the appendix of \cite{DiVizioHardouinWibmer:DifferenceGaloisofDifferential} what difference algebraic groups are, establish the $\s$-Galois correspondence, and show that the $\s$-dimension of the $\s$-Galois group coincides with the $\s$-dimension of the $\s$-\PV\ extension in \S\ref{sec:Galoiscorrespondence}. The relation between isomonodromic difference equations and our Galois theory is given in~\S\ref{sec:isomonodromic}. Applications to difference and differential algebraic properties of functions, including functions with a slow Nevanlinna growth order, and illustrative examples are given in \S\ref{sec:examples}.

We are grateful to S.~Abramov, Y.~Andr\'e, D.~Drasin, D.~Khmelnov, S.~Merenkov, J.~Roques,  and D.~Trushin for their helpful suggestions.
A. Ovchinnikov was supported by the grants: NSF  CCF-0952591 and  PSC-CUNY  No.~60001-40~41.

\section{$\sigma$-\PV\ extensions}

\subsection{Basic definitions and preliminaries}\label{sec:basicdefs}

We need to introduce some terminology from difference algebra. Standard references for difference algebra are \cite{Cohn:difference} and \cite{Levin:difference}.
All rings are assumed to be commutative. By a \emph{$\f$-ring}, we mean a ring $R$ equipped with a ring endomorphism $\f\colon R\to R$. We do not require that $\f$ is an automorphism. If $\f$ is an automorphism, we say that $R$ is \emph{inversive}. By a \emph{$\fs$-ring}, we mean a ring equipped with two commuting endomorphisms $\f$ and $\s$. A morphism of $\f$-rings (or $\fs$-rings) is a morphism of rings that commutes with the endomorphisms. If the underlying ring is a field, we speak of \emph{$\f$-fields} (or \emph{$\fs$-fields}). Here are some basic examples of $\fs$-fields of interest to us:
\begin{ex} \label{ex:fsfields}\hspace{0cm}
\begin{enumerate}
\item
the $\fs$-field $M$ of meromorphic functions on $\C$ with $\f(f)(z) = f(z+z_\f)$ and $\s(f)(z) = f(z+z_\s)$, $f \in M$, $z_\f, z_\s\in\C$ and its $\fs$-subfields $\C(z)$ and $M_{<1}$, the field of meromorphic functions on $\C$ with Nevanlinna growth order less than one (see \S\ref{sec:Nev}),
\item
the $\fs$-field $M$ with $\f(f)(z) = f(z\cdot q_\f)$ and $\s(f)(z) = f(z\cdot q_\s)$, $f \in M$, $q_\f, q_\s\in\C^\times$ and its subfields $\C(z)$ and $M_{<1}$,
\item the $\fs$-field $\C(z,w)$ with $\f(f)(z,w) = f(z+z_\f,w)$ and $\s(f)(z,w) = f(z,w+w_\s)$, $f \in \C(z,w)$, $z_\f, w_\s\in\C$ and various other actions of $\f$ and $\s$ that commute.
\end{enumerate}
\end{ex}

A \emph{$\f$-ideal} in a $\f$-ring $R$ is an ideal $\ida$ of $R$ such that $\f(\ida)\subset \ida$. Similarly, one defines \emph{$\fs$-ideals} in $\fs$-rings.
A $\f$-ring is called \emph{$\f$-simple} if the zero ideal and the whole ring are the only $\f$-ideals. A $\f$-ideal $\q$ in a $\f$-ring $R$ is called $\f$-prime if $\q$ is a prime ideal of $R$ and $\f^{-1}(\q)=\q$. If $\f$ is an endomorphism of a ring $R$, then $\f^d$ is also an endomorphism of $R$ for every $d\geq 1$, and we can speak of $\f^d$-prime ideals of $R$.
A $\f$-ring $R$ is called a $\f$-domain if its zero ideal is $\f$-prime. (Equivalently, $R$ is an integral domain and $\f\colon R\to R$ is injective.) A $\f$-ideal in a $\f$-ring $R$ is called \emph{$\f$-maximal} if it is a maximal element in the set of all $\f$-ideals of $R$, not equal to $R$, ordered by inclusion.

The theory of difference fields does exhibit some pathologies. For example, two extensions of the same difference field can
be incompatible, see \cite[Ch.~5]{Levin:difference}. As it has been recognized in \cite{SingerPut:difference}, the Galois theory of linear difference equations runs much smoother if one allows certain finite products of fields instead of fields. In this context, the following definition has turned out to be useful.

\begin{defi} \label{defi: fpseudofield}
A \emph{$\f$-pseudo field} is a $\f$-simple, Noetherian $\f$-ring $K$ such that every non-zero divisor of $K$ is invertible in $K$.
\end{defi}

The concept of $\f$-pseudo fields (in certain variants) is also used in \cite{AmanoMasuoka:artiniansimple,Wibmer:thesis,Wibmer:Chevalley,Trushin:DifferenceNullstellensatz,Trushin:DifferenceNullstellsatzCaseOfFiniteGroup,HardouinSinger:DifferentialGaloisTheoryofLinearDifferenceEquations,OvchinnikovTrushinetal:GaloisTheory_of_difference_equations_with_periodic_parameters}.

If $K$ is a $\f$-pseudo field, then there exist orthogonal, idempotent elements $e_1,\ldots, e_d$ of $K$ such that
\begin{enumerate}
 \item $K=e_1\cdot K\oplus\cdots\oplus e_d\cdot K$,
\item $\f(e_1)=e_2,\ \f(e_2)=e_3, \ldots, \f(e_d)=e_1$ and
\item $e_i\cdot K$ is a field for $i=1,\ldots,d$ (so, $e_i\cdot K$ is a $\f^d$-field)
\end{enumerate}
(see, for example, \cite[Prop. 1.3.2, p. 9]{Wibmer:thesis}). The integer $d$ is called the \emph{period} of $K$ and denoted by $\per(K)$.
A $\phi$-ideal $\p$ of a $\phi$-ring $R$ is called \emph{$\f$-pseudo prime} if it is the kernel of a morphism from $R$ into some $\f$-pseudo field. Equivalently, $\p$ is of the form
\[\p=\q\cap\phi^{-1}(\q)\cap\ldots\cap\f^{-(d-1)}(\q)\]
for some $\f^d$-prime ideal $\q$ of $R$. We also call $d$ the period of $\p$ (provided that $\q$ is not $\s^{d'}$-prime for some $d'<d$). By a \emph{$\f$-pseudo domain}, we mean a $\f$-ring whose zero ideal is $\f$-pseudo prime. If $R$ is a $\f$-pseudo domain, the period of the zero ideal of $R$ is also called the period of $R$. The total ring of fractions of a $\f$-pseudo domain is a $\f$-ring in a natural way, indeed it is a $\f$-pseudo field. A $\fs$-ring $R$ is called a \emph{$\f$-pseudo $\s$-domain} if $(R,\f)$ is a $\f$-pseudo domain.

\begin{defi}
A $\fs$-ring $K$ is called a \emph{$\f$-pseudo $\s$-field} if $(K,\f)$ is a $\f$-pseudo field.
\end{defi}

Most of the employed nomenclature is self-explanatory. For example, a $K$-$\fs$-algebra is a $K$-algebra $R$ equipped with the structure of a $\fs$-ring such that the $K$-algebra structure map $K\to R$ is a morphism of $\fs$-rings. Constants are denoted by upper indices. For example, if $R$ is a $\f$-ring, then the $\f$-constants of $R$ are
\[R^\f:=\{r\in R|\ \f(r)=r\}.\]
If $K$ is a \fpsf, then $K^\f$ is a $\s$-field (as $R^\f$ is a field for any $\f$-simple $\f$-ring $R$ \cite[Lem.~1.7a), p.~6]{SingerPut:difference}.)

If $R$ is a ring, we denote  the total quotient ring of $R$, i.e., the localization of $R$ at the multiplicatively closed subset of all non-zero divisors, by $\Q(R)$.
If $K$ is a ring, $R$ a $K$-algebra, and $S$ a subset of $R$, then $K(S)$ denotes the smallest $K$-subalgebra of $R$ that contains $S$ and is closed under taking inverses.  So, explicitly \[K(S)=\left\{a/b\:|\:\ a\in K[S],\ b\in K[S]\cap R^\times\right\}\subset R.\]

If $K$ is a $\s$-ring, $R$ a $K$-$\s$-algebra, and $S$ a subset of $R$, then $K\{S\}_\s$ denotes the $K$-$\s$-subalgebra of $R$ generated by $S$, i.e., the $K$-subalgebra of $R$ generated by all elements of the form $\s^d(s)$ where $s\in S$ and $d\geq 0$. (By definition, $\s^0$ is the identity map.) If $R=K\{S\}_\s$ with $S$ finite, we say that $R$ is finitely $\s$-generated over $K$.
Let
\begin{equation}\label{eq:defanglebrackets}
K\langle S\rangle_\s:=\left\{a/b\:|\:\ a\in K\{S\}_\s,\ b\in K\{S\}_\s\cap R^\times\right\}\subset R.
\end{equation}
If $L|K$ is an extension of $\s$-pseudo fields, we say that $L$ is finitely $\s$-generated over $K$ if there exists a finite subset $S$ of $L$ such that
$K\langle S\rangle_\s=L$.

Tensor products of difference rings are considered as difference rings in a natural fashion. For example, if $R$ is a $\f$-ring and $S$, $T$ are $R$-$\f$-algebras then $S\otimes_RT$ becomes an $R$-$\f$-algebra by setting $\f(s\otimes t)=\f(s)\otimes\f(t)$.

Finally, we record some simple and well-known lemmas which we shall use repeatedly throughout the text.

\begin{lemma}{\cite[Lem.~1.1.5, p.~4]{Wibmer:thesis}} \label{lemma: constants of fsimple ring}
Let $R$ be a $\f$-simple $\f$-ring. Then $\Q(R)^\f=R^\f$.
\end{lemma}

\begin{lemma} \label{lemma: correspondence for constants} Let $R$ be a $\f$-simple $\f$-ring and $D$ a $R^\f$-algebra (considered as constant $\f$-ring). The map
$\idb\mapsto R\otimes_{R^\f} \idb$ defines a bijection between the set of all ideals in $D$ and the set of all $\f$-ideals in $R\otimes_{R^\f}D$. The inverse map is given by $\ida\mapsto \ida\cap D$.
\end{lemma}
\begin{proof}
In \cite[Prop.~1.4.15, p.~15]{Wibmer:thesis}, this is stated for the case that $R$ is a $\f$-pseudo field. However, the proof given there only uses the assumption that $R$ is $\f$-simple.
\end{proof}

\begin{lemma} \label{lemma: stays simple after constant extension}
Let $R$ be a $\f$-simple $\f$-ring and $D$ a ($\f$-constant) field extension of $R^\f$. Then $R\otimes_{R^\f}D$ is $\f$-simple.
\end{lemma}
\begin{proof}
This is clear from Lemma \ref{lemma: correspondence for constants}.
\end{proof}

\begin{lemma}{\cite[Lem.~1.1.6, p.~4]{Wibmer:thesis}} \label{lemma: linearly disjoint}
Let $K$ be a $\f$-simple $\f$-ring and $R$ a $K$-$\f$-algebra. Then $K$ and $R^\f$ are linearly disjoint over $K^\f$.
\end{lemma}

\begin{lemma} \label{lemma:finite intersection of sigmapseudoprime ideals}
Let $R$ be a $\f$-simple $\fs$-ring that is a $\f$-pseudo domain. Then $\s$ is injective on $R$ and the zero ideal of $R$ is the finite intersection of $\s$-pseudo prime ideals. Moreover, $\Q(R)$ is naturally a \fpsf.
\end{lemma}
\begin{proof}
Since $\f$ and $\s$ commute, the kernel of $\s$ is a $\f$-ideal. Therefore, $\s$ must be injective. Since $R$ is a $\f$-pseudo domain, the zero ideal of $R$ is a finite intersection of prime ideals. As $\s$ is injective, the map $\q\mapsto\s^{-1}(\q)$ is a permutation of the set of minimal prime ideals of $R$. Every cycle in the cycle decomposition of this permutation corresponds to a $\s$-pseudo prime ideal. Since $R$ is a finite direct sum of integral domains
(\cite[Prop.~1.1.2, p.~2]{Wibmer:thesis}), it is clear that $\s$ and $\f$ extend to $\Q(R)$.
\end{proof}

\subsection{Review of the classical \PV\ theory}\label{sec:classicalPV}
In order to maximize the applicability of our $\s$-Galois theory, we have been careful to avoid unnecessary technical conditions on the base field: 
\begin{enumerate}
\item we work in arbitrary characteristic, 
\item we do not assume that our endomorphisms are automorphisms, and 
\item we do not make any initial requirements on the constants.
\end{enumerate}
Unfortunately, the assumptions in the standard presentations of the classical Galois theory of linear difference equations (e.g. \cite{SingerPut:difference}) are somewhat more restrictive. Since, at some points in the development of our $\s$-Galois theory, we need to use the classical Galois theory, we have to give the definitions and recall the results in our slightly more general setup. This review of the classical theory will also help the reader see the analogy between the classical Galois theory and the $\s$-Galois theory.

\begin{defi} \label{defi: PVext and PV ring}
Let $K$ be a $\f$-pseudo field and $A\in\Gl_n(K)$. An extension $L|K$ of $\f$-pseudo fields with $L^\f=K^\f$ is called a \emph{Picard--Vessiot (\PV) extension} for
$\f(y)=Ay$ if there exists a matrix $Y\in\Gl_n(L)$ such that $\f(Y)=AY$ and $L=K(Y):=K(Y_{ij}|\ 1\leq i,j\leq n)$.

A $\f$-simple $K$-$\f$-algebra $R$ is called a \emph{\PV\ ring} for $\f(y)=Ay$ if there exists $Y\in\Gl_n(R)$ such that $\f(Y)=AY$ and $R=K{\left[Y,1/\det(Y)\right]}$.
\end{defi}

It is easy to describe a construction of a \PV\ ring. Indeed, let $X$ be the $n\times n$-matrix of indeterminates over $K$. We  turn $K[X,1/\det(X)]$ into a $K$-$\f$-algebra by setting $\f(X)=AX$. Then $$K[X,1/\det(X)]/\m$$ is a \PV\ ring for $\f(y)=Ay$ for every $\f$-maximal $\f$-ideal $\m$ of $K[X,1/\det(X)]$. Moreover, every \PV\ ring for $\f(y)=Ay$ is of this form.

The existence of \PV\ extensions is a more delicate issue, unless we assume that $K^\f$ is algebraically closed. The problem is that a \PV\ ring might contain new constants. The following lemma guarantees that the constants of a \PV\ ring over $K$ are an algebraic field extension of $K^\f$. 

\begin{lemma} \label{lemma: almost no new constants}
Let $K$ be a $\f$-pseudo field and $R$ a $\phi$-simple $K$-$\f$-algebra which is finitely generated as $K$-algebra. Then $R^\phi$ is an algebraic field extension of $K^\phi$.
\end{lemma}
\begin{proof}
This is a slight generalization of \cite[Lem.~1.8, p.~7]{SingerPut:difference}. It also follows from \cite[Prop.~2.11, p.~1389]{Wibmer:Chevalley}.
\end{proof}

The following proposition explains the intimate relation between \PV\ extensions and \PV\ rings:

\begin{prop} \label{prop: compare PVextension and PVring classical}
Let $K$ be a $\f$-pseudo field and $A\in\Gl_n(K)$. Let $R$ be a $K$-$\f$-algebra that is a $\f$-pseudo domain. Assume that $R=K[Y,1/\det(Y)]$ for some $Y\in\Gl_n(R)$ with $\f(Y)=AY$. Then $R$ is $\f$-simple if and only if $\Q(R)^\f$ is algebraic over $K^\f$.
\end{prop}
\begin{proof}
 It is clear from Lemmas \ref{lemma: almost no new constants} and~\ref{lemma: constants of fsimple ring} that $\Q(R)^\f$ is algebraic over $K^\f$ if $R$ is $\f$-simple.
So, we assume that $\Q(R)^\f$ is algebraic over $K^\f$. Indeed, we will first assume that $\Q(R)^\f=K^\f$.
Let $$R'=K[Y',1/\det(Y')]$$ be a \PV\ ring for $\f(y)=Ay$ where $Y'\in\Gl_n(R')$ satisfies $\f(Y')=AY'$. Note that $L:=\Q(R)$ is a $\f$-pseudo field.
The matrix $$\quad Z:=(Y^{-1}\otimes 1) \cdot (1\otimes Y')\in\Gl_n(L\otimes_K R')$$ satisfies $\s(Z)=(AY)^{-1}AY'=Z$. It follows from Lemma \ref{lemma: linearly disjoint} that 
\begin{equation} \label{eq: isom}
L\otimes_KR'=L\otimes_{K^\f}K^\f[Z,1/\det(Z)].\end{equation}
Since $K^\f[Z,1/\det(Z)]$ is finitely generated as $K^\f$-algebra, there exists an algebraic field extension $C$ of $K^\f$ and a $K^\f$-morphism $$\psi\colon K^\f[Z,1/\det(Z)]\to C.$$ Composing the inclusion $R'\to L\otimes_K R'$ with (\ref{eq: isom}) and $\id\otimes \psi$, we obtain a $K$-$\f$-morphism $R'\to L\otimes_{K^\f}C$. Since $R'$ is $\f$-simple, we can identify $R'$ with a subring of $L\otimes_{K^\f}C$.
The two solution matrices $Y$ and $Y'$ in $\Gl_n(L\otimes_{K^\f}C)$ only differ by multiplication by an invertible matrix  
with entries in $C$. Therefore,
\[R\otimes_{K^\f}C=K[Y,1/\det(Y)]\otimes_{K^\f} C=K[Y,1/\det(Y), C]=K[Y', 1/\det(Y'),C]=R'\otimes_{R'^\f}C,\]
by Lemma \ref{lemma: linearly disjoint} again. From Lemma \ref{lemma: stays simple after constant extension}, we know that $R'\otimes_{R'^\f}C$ is $\f$-simple. This implies that $R$ is $\f$-simple, because a non-trivial $\f$-ideal of $R$ would give rise to a non trivial $\f$-ideal of $R\otimes_{K^\f}C$. 

In the general case, we set $\widetilde{K}=K\otimes_{K^\f}L^\f\subset L$.
We claim that $\widetilde{K}$ is a $\f$-pseudo field. We already know from Lemma \ref{lemma: stays simple after constant extension} that $\widetilde{K}$ is $\f$-simple and, since  $L$ is a $\f$-pseudo domain, $\widetilde{K}$ is also a $\f$-pseudo domain. Then $\widetilde{K}$ is a finite direct sum of integral domains $R_i$ \cite[Prop.~1.1.2, p.~2]{Wibmer:thesis}. Since $L^\f$ is algebraic over $K^\f$,  $\widetilde{K}$ is integral over $K$. As $K$ is a direct sum of fields $K_j$, this implies that each $R_i$ is integral over some $K_j$. But, since $R_i$ is an integral domain and $K_j$ a field, $R_i$ must be a field. So, $\widetilde{K}$ is a finite direct sum of fields. Consequently, $\widetilde{K}$ is a $\f$-pseudo field.

From the first part of the proof, it follows that $\widetilde{K}[Y,1/\det(Y)]$ is $\f$-simple.
We have to show that $R=K[Y, 1/\det(Y)]$ is $\f$-simple. Suppose that $\ida\subset R$ is a non-trivial $\f$-ideal of $R$. Since $L^\f$ is algebraic over $K^\f$,  $\widetilde{K}[Y,1/\det(Y)]$ is integral over $R$. Therefore, the ideal $\ida'$ of $\widetilde{K}[Y,1/\det(Y)]$ generated by $\ida$ does not contain $1$ \cite[Prop.~4.15, p. 129]{eisenbud:view}. As $\ida'$ is a $\f$-ideal, this yields a contradiction.
\end{proof}

\begin{cor} \label{cor: relation between classical PV-rings and PV-ext}
Let $K$ be a $\f$-pseudo field and $A\in\Gl_n(K)$. If $L|K$ is a \PV\ extension for $\f(y)=Ay$ with fundamental solution matrix $Y\in\Gl_n(L)$, then
$K{\left[Y,1/\det(Y)\right]}$ is a \PV\ ring for $\f(y)=Ay$. Conversely, if $R$ is a \PV\ ring with $R^\phi=K^\phi$, then $\Q(R)$ is a \PV\ extension for $\f(y)=Ay$.
\end{cor}
\begin{proof}
This is clear from Proposition \ref{prop: compare PVextension and PVring classical} and Lemma \ref{lemma: constants of fsimple ring}.
\end{proof}

\begin{theo} \label{theo: uniqueness classical}
Let $K$ be a $\f$-pseudo field. Let $R_1$ and $R_2$ be two \PV\ rings for the same equation $\phi(y)=Ay$, $A\in\Gl_n(K)$. Then there exists a finite algebraic field extension $\widetilde{k}$ of $k:=K^\phi$, containing $k_1:=R_1^\phi$ and $k_2:=R_2^\phi$ and an isomorphism
\[R_1\otimes_{k_1}\widetilde{k}\simeq R_2\otimes_{k_2}\widetilde{k}\]
of $K\otimes_k\widetilde{k}$-$\f$-algebras.
\end{theo}
\begin{proof}
This is a straightforward generalization of \cite[Prop. 1.9, p. 7]{SingerPut:difference}.
\end{proof}

Of course, the above result immediately gives the uniqueness (up to $K$-$\f$-isomorphisms) of \PV\ extensions provided that $K^\f$ is algebraically closed.

\subsection{$\sigma$-\PV\ extensions and $\sigma$-\PV\ rings}\label{sec:sPVdef}
In this section, we define $\sigma$-\PV\ extensions and $\sigma$-\PV\ rings and clarify the relation between them.
Let $K$ be a \fpsf. We  study a linear difference equation
\[\f(y)=Ay, \quad \text{ where } A\in\Gl_n(K). \]
We are mainly interested in the case when $K$ is a field. Typically, $K$ will be one of the $\fs$-fields from Example~\ref{ex:fsfields}. However, for consistency reasons, we will give all definitions over a general \fpsf.

If $R$ is a $K$-$\fs$-algebra, then a matrix $Y\in\Gl_n(R)$ is called a \emph{fundamental solution matrix for $\phi(y)=Ay$} if $\f(Y)=AY$.

\begin{rem} \label{rem: fundamental solution matrix}
If $Y_1,Y_2\in\Gl_n(R)$ are two fundamental solution matrices for $\phi(y)=Ay$, then there exists a matrix $C\in\Gl_n(R^\f)$ such that $Y_2=Y_1C$.
\end{rem}
\begin{proof}
This follows from the well-known computation $\f\big(Y_1^{-1}Y_2\big)=(AY_1)^{-1}AY_2=Y_1^{-1}Y_2$.
\end{proof}

Let $L$ be a \fpsf extension of $K$ and $Y\in\Gl_n(L)$ a fundamental solution matrix for $\phi(y)=Ay$.
If $L=K\langle Y\rangle_\s$, we say that $L$ is $\s$-generated by $Y$.

\begin{defi}
Let $K$ be a $\f$-pseudo $\s$-field and $A\in\Gl_n(K)$. A \fpsf extension $L$ of $K$ is called a \emph{$\s$-\PV\ extension} (or $\s$-parameterized \PV\ extension in case we need to be more precise) for $\phi(y)=Ay$ if $L^\f=K^\f$ and $L$ is $\s$-generated by a fundamental solution matrix for $\phi(y)=Ay$.

A $K$-$\fs$-algebra $R$ that is a $\f$-pseudo $\s$-domain is called a \emph{$\s$-\PV\ ring} for $\f(y)=Ay$ if $R$ is $\f$-simple and $\s$-generated by a fundamental solution matrix for $\f(y)=Ay$, i.e, $R=K{\left\{Y,1/\det(Y)\right\}}_\s$ for some fundamental solution matrix $Y\in\Gl_n(R)$.
\end{defi}

\begin{rem}\label{rem:subtle}
A Noetherian $\f$-simple $\f$-ring is automatically a $\f$-pseudo domain \cite[Prop.~1.1.2, p.~2]{Wibmer:thesis}. This is why the condition that $R$ should be a $\f$-pseudo domain does not appear in the definition of classical \PV\ rings (Definition \ref{defi: PVext and PV ring}). Here, in the $\s$-parameterized setting, one of the more subtle steps in the existence proof of $\s$-\PV\ rings (or extensions) is to verify the $\f$-pseudo domain property (cf. Corollary~\ref{cor: bound period}.) 
\end{rem}

By a \emph{$\s$-\PV\ extension $L|K$}, we mean a \fpsf extension $L$ of $K$ that is a $\s$-\PV\ extension for some linear $\f$-equation $\f(y)=Ay$, with $A\in\Gl_n(K)$. Similarly for $\s$-\PV\ rings. The $\s$-field of $\phi$-constants of a $\s$-\PV\ extension $L|K$ will usually be denoted by $k$, that is,
\[k:=K^\phi=L^\phi.\]

To clarify the relation between $\s$-\PV\ extensions and $\s$-\PV\ rings, we will use the following important observation.

\begin{lemma} \label{lemma:higherorder}
Let $L|K$ be a $\s$-\PV\ extension for $\f(y)=Ay$ with fundamental solution matrix $Y\in\Gl_n(L)$. Set
\[L_d=K{\big(Y,\s(Y),\ldots,\s^d(Y)\big)}\subset L, \quad d\geq 0.\]
Then $L_d|K$ is a \PV\ extension for the $\f$-linear system $\f(y)=A_dy$, where
\[A_d=\left(\begin{array}{cccc}
 A & 0 & \cdots & 0 \\
 0 & \s(A) & \cdots & 0\\
\vdots & & \ddots & \vdots  \\
0 & \cdots & 0 &\s^d(A)
\end{array}\right)\in\Gl_{n(d+1)}(K).
\]
\end{lemma}
\begin{proof}
Note that $K{\left[Y,\s(Y),\ldots,\s^d(Y)\right]}$ is a $\f$-subring of $L$. Therefore, $K{\left(Y,\s(Y),\ldots,\s^d(Y)\right)}$ is a $\f$-pseudo field by \cite[Lem.~1.3.4, p.~9]{Wibmer:thesis}. Applying $\s^i$ to $\f(Y)=AY$ for $i=0,\ldots,d$ yields $\f{\left(\s^i(Y)\right)}=\s^i(A)\s^i(Y).$ Therefore,
\[Y_d=\left(\begin{array}{cccc}
 Y & 0 & \cdots & 0 \\
 0 & \s(Y) & \cdots & 0\\
\vdots & & \ddots & \vdots  \\
0 & \cdots & 0 &\s^d(Y)
\end{array}\right)\in\Gl_{n(d+1)}(L_d).
\]
is a fundamental solution matrix for $\f(y)=A_dy$. Since $L_d^\f\subset L^\f=K^\f$,  $L_d|K$ is a \PV\ extension for $\f(y)=A_dy$.
\end{proof}
The following proposition is the $\s$-analogue of Corollary~\ref{cor: relation between classical PV-rings and PV-ext}.

\begin{prop} \label{prop: compare sigmaPVring simgaPVextension} 
Let $K$ be a \fpsf and $A\in\Gl_n(K)$.
\begin{enumerate}
\item If $L|K$ is a $\s$-\PV\ extension for $\f(y)=Ay$ with fundamental solution matrix $Y\in\Gl_n(L)$, then\\
$R:=K\{Y,1/\det(Y)\}_\s\subset L$ is a $\s$-\PV\ ring for $\f(y)=Ay$.
\item
Conversely, if $R$ is a $\s$-\PV\ ring for $\f(y)=Ay$ with $R^\f=K^\f$, then $\Q(R)$ is a $\s$-\PV\ extension for $\f(y)=Ay$.
\end{enumerate}
\end{prop}
\begin{proof}
Clearly, $R:=K{\left\{Y,1/\det(Y)\right\}}_\s$ is a $\f$-pseudo domain. So, we only have to show that $R$ is $\f$-simple. We know from Lemma \ref{lemma:higherorder} that
$$L_d:=K{\big(Y,\s(Y),\ldots,\s^d(Y)\big)}\subset L$$ is a \PV\ extension of $(K,\f)$ for every $d\geq 0$. It, thus, follows from Corollary 
\ref{cor: relation between classical PV-rings and PV-ext} that \[R_d:=K\big[Y,\s(Y),\ldots,\s^d(Y),1/\big(\det(Y)\cdot\ldots\cdot\det(\s^d(Y))\big)\big]\subset R\] is a \PV\ ring over $K$. So, $R_d$ is $\f$-simple for every $d\geq 0$. Thus, $R$ must be $\f$-simple.

Now assume that $R$ is a $\s$-\PV\ ring with $R^\phi=K^\phi$. From Lemma \ref{lemma:finite intersection of sigmapseudoprime ideals}, we know that $\Q(R)$ is a \fpsf and, by Lemma \ref{lemma: constants of fsimple ring}, we have $\Q(R)^\f=R^\f=K^\f$.
\end{proof}

\subsection{ Existence of $\s$-\PV\ extensions}\label{sec:existence}

In this section, we will establish the existence of $\s$-\PV\ rings (Theorem~\ref{theo: existence of PV ring}) and $\s$-\PV\ extensions (Corollary~\ref{cor:existence of PV extensions}) for a given linear $\f$-equation $\f(y)=Ay$ under rather mild conditions on the base $\fs$-field $K$. The key idea for the existence proof is the prolongation construction from \cite[Lem.~2.16, p.~1392]{Wibmer:Chevalley}. The differential analogue of this construction has also been recently used to establish the existence of $\partial$-parameterized \PV\ extensions for linear differential or difference equations provided that the constants are algebraically closed (see \cite{Wibmer:existence,DiVizioHardouin:DescentandConfluence}). A more elaborate discussion of the existence of differentially parameterized \PV\ extensions for linear differential equations (including the case of several differential parameters) can be found in \cite{Ovchinkiovetal:ParameterizedPicardVessiotExtensionsandAtiyahExtensions}. 

\begin{rem}\label{rem:subtle2}
The idea of the prolongation construction is easy to explain. Indeed, let $K$ be a $\fs$-field and $A\in\Gl_n(K)$. We would like to construct a $\s$-\PV\ ring or a $\s$-\PV\ extension for $\f(y)=Ay$.
Let $$S := K{\left\{X,1/\det(X)\right\}}_\s$$ be the generic solution ring for $\f(y)=Ay$. By this, we mean that $X$ is the $n\times n$-matrix of $\s$-indeterminates, and the action of $\f$ is determined by $\f(X)=AX$. Finding a $\s$-\PV\ ring for $\f(y)=Ay$ is equivalent to finding a $\fs$-ideal $\m$ of $S$ that is $\f$-pseudo prime and $\f$-maximal. The existence of a $\f$-maximal ideal in $S$ is, of course, guaranteed by Zorn's lemma, but it is unclear if we can find a $\f$-maximal ideal that is additionally a $\s$-ideal and $\f$-pseudo prime.
\end{rem}

If $L$ is a $\s$-\PV\ extension for $\f(y)=Ay$ with fundamental solution matrix $Y\in\Gl_n(L)$, then  
$R_d$ is a \PV\ ring over $K$ for $\f(y)=A_dy$, as we have already seen in Lemma~\ref{lemma:higherorder} and Proposition \ref{prop: compare sigmaPVring simgaPVextension}. Thus, we should better find a $\fs$-ideal $\m$ of $S$ such that $$\m_d:=\m\cap S_d,\quad S_d := K{\big[X,\dots,\s^{d}(X),1/\det\big(X\cdot\ldots\cdot\s^{d}(X)\big)\big]}\subset S$$ is $\f$-maximal in $S_d$ for every $d\geq 0$. Note that not every $\f$-maximal $\f$-ideal of $S_d$ is of the form $\m_d$ for some $\f$-maximal $\fs$-ideal $\m$ of $S$. A necessary condition is given
by $$\s(\m_d\cap S_{d-1})\subset \m_d.$$ However, if we assume that we  have already constructed a $\f$-maximal $\f$-ideal $\m_d$ of $S_d$ that satisfies this condition, we can try to construct $\m_{d+1}$ by a choosing a $\f$-maximal $\f$-ideal of $S_{d+1}$ that contains $\m_d$ and $\s(\m_d)$. Then we could define $\m$ as the union of all the $\m_d$'s.

There are two obstructions to this procedure that we will have to overcome: 
\begin{enumerate} 
\item The ideal of $S_{d+1}$ generated by $\m_d$ and $\s(\m_d)$ might contain $1$. In this case, the construction would not apply.
\item The union $\bigcup\m_d$ is  
a $\f$-maximal $\fs$-ideal, but, a priori, it is unclear why it should be a $\f$-pseudo prime ideal.
\end{enumerate}
The following purely algebraic lemma is the crucial ingredient to overcome the first difficulty. The second difficulty will be resolved in Lemma~\ref{lemma:bound period}, which will eventually provide a bound for the period of $\m_d$.

\begin{lemma} \label{lemma: prolongation}
Let $K$ be a field and let $R$ be a finitely generated $K$-algebra. For $d\geq 0$, let $R_0,\ldots, R_{d+1}$ denote isomorphic copies of $R$.
Let $\ida\subset R_0\otimes\cdots\otimes R_{d}$ and $\idb\subset R_1\otimes\cdots\otimes R_{d+1}$ be ideals not containing $1$. (The tensors are understood to be over $K$.) Assume that
\begin{equation}\label{eq:lem18asmp}
\ida\cap (R_1\otimes\cdots\otimes R_d)=\idb\cap (R_1\otimes\cdots\otimes R_d).
\end{equation} Then the ideal of $R_0\otimes\cdots\otimes R_{d+1}$ generated by $\ida$ and $\idb$ does not contain $1$.
\end{lemma}
\begin{proof}
We set $X_i:=\spec(R_i)$ for $i=0,\ldots,d+1$. Let $Y$ and $Z$ denote the closed subschemes defined by $\ida$ and $\idb$, respectively.
We have to show that \[(Y\times X_{d+1})\cap(X_0\times Z)\subset X_0\times\cdots\times X_{d+1}\]
is non-empty.
Let $\pi_{1d}\colon X_0\times\cdots\times X_{d+1}\to X_1\times\cdots\times X_{d} $ denote the projection onto the factors ``in the middle''.
Assumption~\eqref{eq:lem18asmp} means that 
$$\overline{\pi_{1d}(Y\times X_{d+1})}=\overline{\pi_{1d}(X_0\times Z)} =: W.$$
By Chevalley's theorem, the image of a morphism of schemes of finite type over a field contains a dense open subset of its closure. Thus, there exist open dense subsets $U,V\subset W$ with $$U\subset\pi_{1d}(Y\times X_{d+1})\quad \text{and}\quad V\subset \pi_{1d}(X_0\times Z).$$ Then $U\cap V$ is also dense and open in $W$. In particular,
$$U\cap V\subset \pi_{1d}(Y\times X_{d+1})\cap \pi_{1d}(X_0\times Z)$$ is non-empty. But an element of $\pi_{1d}(Y\times X_{d+1})\cap \pi_{1d}(X_0\times Z)$ gives rise to an element of
$(Y\times X_{d+1})\cap(X_0\times Z)$.
\end{proof}
If $R$ is a $\f$-ring, we denote the ring of its {\em $\f$-periodic elements} by
\[R^{\f^\infty}={\left\{r\in R\:|\: \exists\ m\geq 1 \text{ such that } \f^m(r)=r\right\}}.\]
It is a $\f$-subring of $R$. If $K$ is a $\phi$-field, then $K^{\phi^\infty}$ is the relative algebraic closure of $K^\phi$ in $K$ \cite[Thm~2.1.12, p.~114]{Levin:difference}. In particular, if $K^\f$ is algebraically closed, then $K^{\phi^\infty}=K^\f$.

Analogues of the generic solution field $U$ in the following lemma appear in \cite[\S4]{ChatzidakisHardouinSinger:OntheDefinitionsOfDifferenceGaloisgroups} and \cite{Nieto:OnSigmaDeltaPicardVessiotExtension}. The relation between the periodic elements in a universal solution field and the period of a \PV\ ring, which we shall eventually use to bound the period of $\m_d$, has been found in \cite{ChatzidakisHardouinSinger:OntheDefinitionsOfDifferenceGaloisgroups}. In the language of \cite{ChatzidakisHardouinSinger:OntheDefinitionsOfDifferenceGaloisgroups}, the following lemma essentially says that the $m$-invariant of the systems $\f(y)=A_dy$ is bounded (as a function of $d\geq 0$).

\begin{lemma} \label{lemma:bound period}
Let $K$ be a $\fs$-field such that $K^{\f^\infty}=K^\f$. Let $A\in\Gl_n(K)$ and let $X$ denote the $n\times n$-matrix of $\s$-indeterminates over $K$. Set 
$$
U=K\langle X\rangle_\s(=\Q(K\{X_{ij}\:|\: 1\leq i,j\leq n\}_\s))
$$ and define a $\fs$-structure on $U$ by  $\f(\s^i(X))=\s^i(A)\s^i(X)$, $i\Ge 0$.
Then $U^{\f^\infty}$ is a finite field extension of $U^\f$.
\end{lemma}
\begin{proof}
We have a tower of $\fs$-fields $KU^\f\subset KU^{\f^\infty}\subset U$. By construction, $U$ is a finitely $\s$-generated $\s$-field extension of $KU^\f$. Since an intermediate $\s$-field of a finitely $\s$-generated $\s$-field extension is itself finitely $\s$-generated (\cite[Thm.~4.4.1, p.~292]{Levin:difference}), it follows that $KU^{\f^\infty}$ is finitely $\s$-generated over $KU^\f$. Hence, we can find $$a_1,\ldots,a_m\in U^{\f^\infty}$$ that $\s$-generate $KU^{\f^\infty}$ as a $\s$-field extension of $KU^\f$.
We claim that $$U^{\f^\infty}=U^\f\langle a_1,\ldots,a_m\rangle_\s.$$ The inclusion ``$\supset$'' is clear. So, let $a\in U^{\f^\infty}$. Let $(b_i)_{i\in I}$ be a $K^\f$-basis of $U^\f\langle a_1,\ldots,a_m\rangle_\s$. As $$a\in KU^{\f^\infty}=KU^\f\langle a_1,\ldots,a_m\rangle_\s,$$ we can write
\[a=\frac{\sum\lambda_i\cdot  b_i}{\sum{\mu_i}\cdot b_i}\]
with $\lambda_i,\mu_i\in K$. Multiplying by the denominator yields
\begin{equation} \label{eq:linear dependence}
\sum \mu_i\cdot a\cdot b_i=\sum\lambda_i\cdot b_i.
\end{equation}
We can choose an integer $e\geq 1$ such that $a,b_i\in U^{\f^e}$ whenever $\lambda_i$ or $\mu_i$ is non-zero.
Then~\eqref{eq:linear dependence} signifies that the family $${\left(a\cdot b_i,b_j\right)}_{i,j\in I}\quad \text{in}\quad U^{\f^e}$$ is $K$-linearly dependent. Since $K$ is linearly disjoint from $U^{\f^e}$ over $K^{\f^e}=K^\f$ (Lemma \ref{lemma: linearly disjoint}), we can find a non-trivial relation
\begin{equation} \label{eq:linear dependence2}
\sum \mu_i'\cdot a\cdot b_i=\sum\lambda_i'\cdot b_i
\end{equation}
with $\mu_i',\lambda_i'\in K^\f$. Suppose $\sum\mu_i'b_i=0$. Then also $\sum\lambda_i'b_i=0$.  Since the $b_i$'s are $K^\f$-linearly independent, this is only possible if relation (\ref{eq:linear dependence2}) is trivial. Therefore, we can divide by the denominator to find that
\[a=\frac{\sum\lambda_i'\cdot  b_i}{\sum{\mu_i'}\cdot b_i}\in U^\f\langle a_1,\ldots,a_m\rangle_\s\]
as desired.
Now let $e\geq 1$ be such that $a_1,\ldots,a_m\in U^{\f^e}$. Then it follows from $$U^{\f^\infty}=U^\f\langle a_1,\ldots,a_m\rangle_\s$$ that $U^{\f^\infty}=U^{\f^e}$. Consequently,
${\left[U^{\f^\infty}\colon U^\f\right]}\leq e$.
\end{proof}

\begin{cor} \label{cor: bound period}
Let $K$ be a $\fs$-field such that $K^{\f^\infty}=K^\f$. Let $A\in\Gl_n(K)$. For $d\geq 0$, let $R_d$ be a \PV\ ring for $\f(y)=A_dy$, where
\[A_d=\left(\begin{array}{cccc}
 A & 0 & \cdots & 0 \\
 0 & \s(A) & \cdots & 0\\
\vdots & & \ddots & \vdots  \\
0 & \cdots & 0 &\s^d(A)
\end{array}\right)\in\Gl_{n(d+1)}(K).
\]
Then the sequence ${(\per(R_d))}_{d\geq 0}$ is bounded. 
\end{cor}
\begin{proof}
Let $U=K\langle X\rangle_\s$ as in Lemma \ref{lemma:bound period}. We will show that $\per(R_d)\leq {\left[U^{\f^\infty}\colon U^\f\right]}$ for $d\geq 0$.
Let $\overline{U^\f}$ denote an algebraic closure of $U^\f$, considered as a constant $\f$-ring. We know that $K$ is a regular field extension of $K^\f$. (By assumption, $K^\f$ is relatively algebraically closed in $K$ and $K$ is always separable over $K^\f$ \cite[Cor.~1.4.16, p.~16]{Wibmer:thesis}). Therefore, $K\otimes_{K^\f}\overline{U^\f}$ is an integral domain. Moreover, $K\otimes_{K^\f}\overline{U^\f}$ is $\f$-simple by Lemma \ref{lemma: stays simple after constant extension}. It follows that
\[K':=\Q{\big(K\otimes_{K^\f}\overline{U^\f}\big)}\]
is a $\f$-field with $K'^\f=\overline{U^\f}$ algebraically closed.
It is clear from the definition of $U$ that $$KU^\f{\big(X,\ldots,\s^d(X)\big)}\subset U$$ is a \PV\ extension of $KU^\f$ for the linear $\f$-equation $\f(y)=A_dy$. It follows from Corollary \ref{cor: relation between classical PV-rings and PV-ext} that $$S_d:=KU^\f{\big[X,\ldots,\s^d(X),1/\det\big(X\cdot\ldots\cdot\s^d(X)\big)\big]}$$ is a \PV\ ring over $KU^\f$.
Then $S_d':=S_d\otimes_{U^\f}\overline{U^\f}$ is a \PV\ ring over
\[KU^\f\otimes_{U^\f}\overline{U^\f}=\Q{\left(K\otimes_{K^\f}U^\f\right)}\otimes_{U^\f}\overline{U^\f}=\Q{\big(K\otimes_{K^\f}\overline{U^\f}\big)}=K'\]
by Lemma \ref{lemma: stays simple after constant extension}.
Note that $S_d\subset U$ is an integral domain and that $$\per{\big(S_d'\big)}\leq {\left[U^{\f^\infty}\colon U^\f\right]}$$ as $U^{\f^\infty}$ is the relative algebraic closure of $U^\f$ in $U$.
As $R_d$ is a \PV\ ring for $\f(y)=A_dy$ over $K$,  $R_d\otimes_{R_d^\f} \overline{U^\f}$ is $\f$-simple by
Lemma \ref{lemma: stays simple after constant extension}. (Note that $R_d^\f$ can be embedded in $\overline{U^\f}$ by Proposition \ref{prop: compare PVextension and PVring classical}.)
The canonical map
$$K\otimes_{K^\f} \overline{U^\f}\to R_d\otimes_{R_d^\f} \overline{U^\f}$$
is injective, because $K\otimes_{K^\f}\overline{U^\f}$ is $\f$-simple. Localizing this inclusion
at the non-zero divisors of $K\otimes_{K^\f} \overline{U^\f}$, we obtain a \PV\ ring $R_d'$ over $K'$. Since $K'^\f=\overline{U^\f}$ is algebraically closed, $R_d'$ and $S_d'$ are isomorphic. It follows that
\[\per(R_d)\leq\per{\left(R_d'\right)}=\per{\left(S_d'\right)}\leq {\left[U^{\f^\infty}\colon U^\f\right]}.\let\qedsymbol\openbox\qedhere
\]
\end{proof}

We are now prepared to establish the main existence theorem.

\begin{theo} \label{theo: existence of PV ring}
Let $K$ be a $\fs$-field such that $K^{\f^\infty}=K^\f$, $\s\colon K^\f\to  K^\f$ is an automorphism, and $A\in\Gl_n(K)$. Then there exists a $\s$-\PV\ ring $R$ for $\f(y)=Ay$ such that $R^\f$ is an algebraic field extension of $K^\f$.
\end{theo}
\begin{proof}
We first assume that $\s\colon K\to K$ is an automorphism. Let $X$ be the $n\times n$-matrix of $\s$-indeterminates over $K$. We denote  the localization of the $\s$-polynomial ring $K{\{X_{ij}|\ 1\leq i,j\leq n\}}_\s$ at the multiplicatively closed subset generated by $\det(X),\s(\det(X)),\ldots$ by $S$. This is naturally a $K$-$\s$-algebra. We define a $\fs$-structure on $S$ by setting
$$
 \f(X)=AX,\quad
\f(\s(X))=\s(A)\s(X),\quad
\f{\left(\s^2(X)\right)}=\s^2(A)\s^2(X),\quad
\ldots
$$
For $0\leq i\leq j$, we also define the following $K$-$\f$-subalgebras of $S$:
\[S_{i,j}=K{\left[\s^i(X),\tfrac{1}{\s^i(\det(X))},\ldots,\s^j(X), \tfrac{1}{\s^j(\det(X))}\right]}=K{\left[\s^i(X),\ldots,\s^j(X),\tfrac{1}{\det(\s^i(X)\cdot\ldots\cdot\s^j(X))}\right]}\subset S,\quad S_j := S_{0,j}.\]
We will show by induction on $d\geq 0$ that there exists a sequence ${(\m_d)}_{d\geq 0}$ with the following properties:
\begin{enumerate}
 \item $\m_d$ is a $\f$-maximal $\f$-ideal of $S_d$,
\item $\m_d\cap S_{d-1}=\m_{d-1}$, and
\item $\s^{-1}(\m_d)=\m_{d-1}$, where $\s\colon S_{d-1}\to S_{d}$.
 \end{enumerate}
 For $d=0$, we can choose $\m_0$ to be any $\f$-maximal $\f$-ideal of $S_0=K[X,1/\det(X)]$. Assume that a sequence $\m_0,\ldots,\m_{d}$ with the desired properties has  been already constructed. We will construct $\m_{d+1}$.
Let $\ida$ denote the ideal of $S_{d+1}$ generated by $\m_{d}$ and $\s(\m_{d})$. The crucial step now is to show that $1\notin\ida$. For this, we would like to apply Lemma \ref{lemma: prolongation}. Note that $S_{d+1}$ is the $d+2$-fold tensor product of $S_0$ with itself. Since $\s$ is an automorphism on $K$, 
$$
\s\colon S_d\to S_{1,d+1}
$$ is an isomorphism and so $\s(\m_d)$ is an ideal of $S_{1,d+1}$.
We need to verify that
\[\m_d\cap S_{1,d}=\s(\m_d)\cap S_{1,d}.\]
Let $f\in \m_d\cap S_{1,d}$. Then $f$ is of the form $f=\s(g)$ for
some $g\in S_{d-1}$. Since $f\in\m_d$, we have $$g\in\s^{-1}(\m_d)=\m_{d-1}\subset\m_d.$$ Thus, $f\in\s(\m_d)$.
Now let $f\in \s(\m_d)\cap S_{1,d}$. Then $f$ is of the form $f=\s(g)$ with
$$g\in\m_d\cap S_{d-1}=\m_{d-1}.$$ So $f=\s(g)\in\m_d$.
We can thus apply Lemma \ref{lemma: prolongation} to conclude that $1\notin\ida$. By construction, $\ida$ is a $\f$-ideal of $S_{d+1}$. Let $\m_{d+1}$ be a $\f$-maximal $\f$-ideal of $S_{d+1}$ containing $\ida$. Then $$\m_{d+1}\cap S_d\quad \text{and}\quad \s^{-1}(\m_{d+1})$$ are $\f$-ideals of $S_{d}$ containing $\m_d$. As $\m_d$ is $\f$-maximal in $S_d$, it follows that $$\m_{d+1}\cap S_d=\m_d\quad \text{and}\quad \s^{-1}(\m_{d+1})=\m_d.$$ This concludes the inductive step.
Now that we have constructed the sequence ${(\m_d)}_{d\geq 0}$, we can define
\[\m:=\bigcup\nolimits_{d\geq 0}\m_d.\]
This is a $\fs$-ideal of $S$. Since the $\m_d$'s are $\f$-maximal, it follows that $\m$ is also $\f$-maximal.
The next crucial step is to show that $\m$ is $\f$-pseudo prime.

In general, a $\f$-maximal $\f$-ideal need not be $\f$-pseudo prime. However, a $\f$-maximal $\f$-ideal that has only finitely many minimal prime ideals is $\f$-pseudo prime \cite[Prop. 1.1.2, p. 2]{Wibmer:thesis}. In particular, in a Noetherian $\f$-ring, every $\f$-maximal $\f$-ideal is $\f$-pseudo prime. So the $\m_d$'s are $\f$-pseudo prime ideals. Thus, to show that $\m$ is $\f$-pseudo prime, it will suffice to show that the sequence ${(\per(\m_d))}_{d\geq 0}$ is bounded.
But this is clear from Corollary \ref{cor: bound period}, because $R_d:=S_d/\m_d$ is a \PV\ ring for $\f(y)=A_dy$.

So, $R:=S/\m$ is a $\f$-pseudo domain, and it is clear from the construction that $R$ is a $\s$-\PV\ ring for $\f(y)=Ay$ over $K$. It remains to see that $R^\f$ is algebraic over $K^\f$. But $R$ is the union of the $R_d$'s and the $R_d$'s are \PV\ rings over $K$, so $R_d^\f$ is algebraic over $K^\f$ (Lemma \ref{lemma: almost no new constants}) and, consequently, $R^\f$ is algebraic over $K^\f$. This concludes the proof for the case that $\s\colon K\to K$ is surjective.

Now let $\s\colon K\to K$ be arbitrary. We consider the inversive closure $K^*$ of $K$ with respect to $\sigma$ (see \cite[Def.~2.1.6, p.~109]{Levin:difference}.) For every $a\in K^*$, there exists an integer $l\geq 1$ such that $\s^l(a)\in K$. We naturally extend $\f$ from $K$ to $K^*$ by $$\f(a)=\s^{-l}{\big(\f{\big(\s^l(a)\big)}\big)}.$$ 
Suppose that $a\in {K^*}^{\f^d}$.
Then $$a=\f^d(a)=\s^{-l}{\big(\f^d{\big(\s^l(a)\big)}\big)}$$ and so $$\s^l(a)=\f^d{\big(\s^l(a)\big)},$$ that is, $\s^l(a)\in K^{\f^d}=K^\f$. By the hypothesis, $K^\f$ is $\s$-inversive. Therefore, $a\in K^\f$. It follows that $${K^*}^{\f^\infty}=K^\f={K^*}^\f.$$
By the first part of the proof, there exists a $\s$-\PV\ ring $R^*$ over $K^*$ for $\f(y)=Ay$ with ${R^*}^\f$ algebraic over $K^\f$. Let $Y\in\Gl_n(R^*)$ denote a fundamental matrix.
We claim that $$R:=K{\left\{Y,1/\det(Y)\right\}}_\s\subset R^*$$ is a $\s$-\PV\ ring for $\f(y)=Ay$ over $K$ with $R^\f$ algebraic over $K^\f$.
As ${R^*}^\f$ is algebraic over ${K^*}^\f=K^\f$, $R^\f$ is algebraic over $K^\f$. So it only remains to show that $R$ is $\f$-simple. For this, it suffices to show that
$$R_d:=K\big[Y,1/\det(Y),\ldots,\s^d(Y),1/\det{\big(\s^d(Y)\big)}\big]$$ is $\f$-simple for every $d\geq 0$.
Let $L^*$ denote the total quotient ring of $R^*$ and $L_d$ the total quotient ring of $R_d$. Since $R^*$ is $\f$-simple, we have ${L^*}^\f={R^*}^\f$ by Lemma \ref{lemma: constants of fsimple ring}. As $L_d\subset L^*$, it follows that $L_d^\f$ is algebraic over $K^\f$. By Proposition~\ref{prop: compare PVextension and PVring classical}, this implies that $R_d$ is $\f$-simple.
\end{proof}

\begin{cor}[Existence of $\s$-\PV\ extensions] \label{cor:existence of PV extensions}
Let $K$ be a $\fs$-field and $A\in\Gl_n(K)$. Assume that $K^\f$ is an algebraically closed inversive $\s$-field. Then there exists a $\s$-\PV\ extension for $\f(y)=Ay$.
\end{cor}
\begin{proof}
This is clear from Theorem \ref{theo: existence of PV ring} and Proposition \ref{prop: compare sigmaPVring simgaPVextension}.
\end{proof}

\subsection{Existence of $\s$-\PV\ extensions for some specific base fields} \label{subsec: special existence}

The purpose of this section is to establish the existence of $\s$-\PV\ extensions over important $\fs$-fields like $K=\mathbb{C}(t,z)$, where
$$\f(f(t,z))=f(t,z+1)\quad \text{and}\quad \s(f(t,z))=f(qt,z)\quad \text{or}\quad \s(f(t,z))=f(t+\alpha,z)$$ for some $q,\alpha\in\C^\times$. Note that the general existence result for $\s$-\PV\ extensions  (Corrollary \ref{cor:existence of PV extensions}) does not apply because $K^\f=\mathbb{C}(t)$ is not algebraically closed.

We will show quite generally that, for every linear $\f$-equation $\f(y)=Ay$ over $K=k(z)$, there exists a $\s$-\PV\ extension, where $k$ is an arbitrary $\s$-field of characteristic zero. Moreover, we give a very concrete recipe how $\s$-\PV\ rings over such $K$ can be constructed inside rings of sequences. Cf. \cite[Prop. 4.1, p. 45]{SingerPut:difference}.

Let $k$ be a field. The ring $\seq_k$ of sequences in $k$ (cf. \cite[Ex.~1.3, p. 4]{SingerPut:difference}) consists of all sequences
\[a=(a(0),a(1),\ldots),\quad a(0),a(1),\ldots\in k,\]
and two sequences are identified if they agree starting from some index. The ring structure of $\seq_k$ is given by the componentwise addition and multiplication. By setting
\[\f((a(0),a(1),a(2),\ldots))=(a(1),a(2),\ldots),\]
we turn $\seq_k$ into an inversive $\f$-ring.
If $k$ is a $\s$-field, then $\seq_k$ naturally becomes  a $\fs$-ring by setting
\[\s((a(0),a(1),\ldots))=(\s(a(0)),\s(a(1)),\ldots).\]
Note that $\seq_k^\f=k$.
We consider $k(z)$, the field of rational function in one variable over $k$, as $\fs$-field by setting $$\f(f(z))=f(z+1),\ \ f\in k(z),\quad \text{and}\quad \s(z)=z.$$ If $\Char k =0$, we can define a $\fs$-embedding $$k(z)\to\seq_k\quad\text{by}\quad f\mapsto (f(0),f(1),\ldots).$$
The expression $f(i)$ is well-defined for $i\gg0$, as the denominator of $f\in k(z)$ has only finitely many zeros.

\begin{prop} \label{prop:special existence}
Let $k$ be a $\s$-field of characteristic zero  
and consider $K=k(z)$ as a $\fs$-field via $$\f(f(z))=f(z+1)\quad \text{and}\quad \s(z)=z.$$
Let $A\in\Gl_n(K)$ and  $i_0\geq 0$ be an integer such that $A(i)$ is well-defined and $\det(A(i))\neq 0$ for all $i\geq i_0$.  Define $Y\in\Gl_n(\seq_k)$ by 
$$Y(i_0)=\id\quad \text{and}\quad Y(i)=A(i-1)Y(i-1),\quad i>i_0.$$ Then $Y$ is a fundamental solution matrix for $\f(y)=Ay$ and $$K{\left\{Y,1/\det(Y)\right\}}_\s\subset\seq_k$$ is a $\s$-\PV\ ring for $\f(y)=Ay$. Moreover, there exists a $\s$-\PV\ extension for $\f(y)=Ay$.
\end{prop}
\begin{proof}
It is clear that $Y$ is a fundamental solution matrix and that $$R:=K{\left\{Y,1/\det(Y)\right\}}_\s$$ is a $\fs$-ring. It remains to see that $R$ is a $\f$-simple $\f$-pseudo domain. To see that $R$ is $\f$-simple, it suffices to show that $$R_d:=K{\big[Y,\ldots,\s^d(Y),1/\det{\big(Y\cdots\s^d(Y)\big)}\big]}$$ is $\f$-simple for every $d\geq 0$.
Note that
\[Y_d=\left(\begin{array}{cccc}
 Y & 0 & \cdots & 0 \\
 0 & \s(Y) & \cdots & 0\\
\vdots & & \ddots & \vdots  \\
0 & \cdots & 0 &\s^d(Y)
\end{array}\right)\in\Gl_{n(d+1)}(R).
\]
is a fundamental solution matrix for $\f(y)=A_dy$ (cf. Lemma \ref{lemma:higherorder}.) By \cite[Prop. 2.4, p. 4]{Wibmer:Skolem-Mahler-Lech}, there exists a \PV\ ring $S_d$ for $\f(y)=A_dy$ over $K$ inside $\seq_k$. As $\seq_k^\f=k$ and two fundamental solution matrices for the same equation only differ by multiplication by a matrix with constant entries, it follows that $R_d=S_d$. In particular, $R_d$ is $\f$-simple.
It follows from Corollary \ref{cor: bound period} that $R$ is a $\f$-pseudo domain. 

As $\seq_k^\f=k=K^\f$, Proposition \ref{prop: compare sigmaPVring simgaPVextension} implies that $\quot(R)$ is a $\s$-PV extension for $\f(y)=Ay$.
\end{proof}

\begin{rem}
Let $Y\in\Gl_n(\seq_k)$ be defined as in Proposition \ref{prop:special existence}. It is unclear whether or not $K\langle Y\rangle_\s\subset\seq_k$ (see~\eqref{eq:defanglebrackets}) is a $\s$-\PV\ extension for $\f(y)=Ay$. The difficulty here is to know that a non-zero divisor of $K{\left\{Y,1/\det(Y)\right\}}_\s\subset\seq_k$ is a unit in $\seq_k$. This problem is closely related to the generalization of the Skolem--Mahler--Lech theorem to rational function coefficients (see \cite{Wibmer:Skolem-Mahler-Lech}). It follows from \cite[Cor. 3.4, p. 8]{Wibmer:Skolem-Mahler-Lech} that $K\langle Y\rangle_\s\subset\seq_k$ is a $\s$-\PV\ extension for $\f(y)=Ay$ if $A\in\Gl_n(k[z])$.
\end{rem}

\subsection{Uniqueness}\label{sec:uniqueness}
In this section, we will establish the uniqueness of $\s$-\PV\ rings and $\s$-\PV\ extensions (for a given equation $\f(y)=Ay$). In other words, we prove a result analogous to the classical uniqueness theorem (Theorem \ref{theo: uniqueness classical}).
The main difficulty is to understand what the $\s$-analogue of the algebraic closure in the classical case is. There is a notion of a difference closed difference field that has been used and studied extensively by model theorists (see, for example, \cite{Hrushovskietal:ModelTheoryofDifferencefields,Hrushovskietal:ModelTheoryofDifferenceFieldsIIPeriodicIdelas}).

\begin{defi}\label{defi: sclosed sfield}
A $\s$-field $k$ is called \emph{$\s$-closed} if for every finitely $\s$-generated $k$-$\s$-algebra $R$ which is a $\s$-domain, there exists a $k$-$\s$-morphism $R\to k$.
\end{defi}

In contrast to differential algebra, there appears to be no satisfactory notion of a $\s$-closure of a $\s$-field. Kolchin preferred the term ``constrainedly closed'' to ``differentially closed'' because a differentially closed differential field can have proper differential algebraic extensions.
The following definition can be seen as an adaptation of Kolchin's notion of constrained extensions of differential fields (\cite{Kolchin:constrainedExtensions}) to difference algebra.

\begin{defi} \label{defi: constrained}
Let $L|K$ be an extension of $\s$-pseudo fields. We say that $L$ is \emph{constrained} over $K$ if, for every finite tuple $a$ from $L$, there exists a non-zero divisor $b\in L$ such that $(0)$  is the only $\s$-pseudo prime ideal of $K\{a,1/b\}_\s$.
\end{defi}

The basic properties of constrained extensions of $\s$-pseudo fields have been established in \cite[\S2.1]{Wibmer:Chevalley}. The relation to $\s$-closed $\s$-fields is given by the fact that a $\s$-closed $\s$-field does not have proper constrained $\s$-field extensions. More generally, every finitely $\s$-generated $\s$-pseudo field extension of a $\s$-closed $\s$-field $k$ is of the form
$k\oplus\cdots\oplus k$ (see \cite[Ex.~2.8, p.~1388]{Wibmer:Chevalley}).

The following theorem is the crucial tool from difference algebra for proving our uniqueness result. It can be seen as a difference analogue of a theorem of Chevalley. For a prime ideal $\q$ in a $\s$-ring $R$ and $r \in R$, we write $$r\notin_\s\q$$ if $\s^d(r)\notin\q$ for every $d\geq 0$.

\begin{theo} \label{theo:sigmaChevalley}
Let $R\subset S$ be an inclusion of $\s$-rings such that $S$ is finitely $\s$-generated over $R$. Assume that $R$ is a $\s$-domain and   $(0) \subset S$ is a finite intersection of $\s$-pseudo prime ideals. Then there exist  $0\ne r\in R$ and an integer $l\geq 1$ such that, for every $d\geq 1$ and  $\s^d$-prime ideal $\q$ of $R$ with $r\notin_\s\q$, there exists a $\s^{ld}$-prime ideal $\q'$ of $S$ with $\q'\cap R=\q$.
\end{theo}
\begin{proof}
This is a slight generalization of \cite[Thm.~1.15, p.~1384]{Wibmer:Chevalley}, where it is assumed that $S$ is a $\s$-domain. There exists a minimal prime ideal $\widehat{\q}$ of $S$ with $\widehat{\q}\cap R =(0)$ \cite[Ch.~II, \S 2, Sec.~6, Prop.~16, p.~74]{Bourbaki:commutativealgebra}. By assumption, $\widehat{\q}$ is a $\s^{\widehat{d}}$-prime ideal for some $\widehat{d}\geq 1$. We can now apply \cite[Thm.~1.15, p.~1384]{Wibmer:Chevalley} to the inclusion
$R\subset S/\widehat{\q}$ of $\s^{\widehat{d}}$-domains to obtain $0\ne r \in R$ and an integer $\widehat{l}\geq 1$ such that, for every $\s^{d\widehat{d}}$-prime ideal $\q$ of $R$ with $r\notin_{\s^{\widehat{d}}}\q$, there exists a $\s^{\widehat{l}d\widehat{d}}$-prime ideal $\q'$ of $S/\widehat{\q}$ with $\q'\cap R=\q$. Set $l:=\widehat{d}\widehat{l}$. Observing that a $\s^{d}$-prime ideal is a $\s^{d\widehat{d}}$-prime ideal and that $r\notin_{\s}\q$ implies $r\notin_{\s^{\widehat{d}}}\q$ yields the claim of the theorem.
\end{proof}

We will need a few more preparatory results.

\begin{lemma} \label{lemma: stays linearly independent}
Let $k$ be an inversive $\s$-field and $R$ a $k$-$\s$-algebra with $\s\colon R\to R$ injective. If $(\lambda_i)$ is a family of $k$-linearly independent elements from $R$, then  the family $(\sigma(\lambda_i))$ is $k$-linearly independent as well. 
\end{lemma}
\begin{proof}
If $\sum a_i\s(\lambda_i)=0$ with $a_i\in k$, then, as $k$ is inversive, we can find $b_i\in k$ with $\s(b_i)=a_i$. We have $\s(\sum b_i\lambda_i)=0$, and this implies $\sum b_i\lambda_i=0$. Therfore, the $b_i$'s and also the $a_i$'s are all zeroes.
\end{proof}

\begin{lemma} \label{lemma: stays injective}
Let $k$ be an inversive $\s$-field and $R$ a $k$-$\s$-algebra with $\s\colon R\to R$ injective. Then $\s\colon R\otimes_k K\to R\otimes_kK$ is injective for every $\s$-field extension $K$ of $k$. Moreover, if $\ida$ is a reflexive $\s$-ideal of $R$ (i.e., $\s^{-1}(\ida)=\ida$), then $\ida\otimes_k K$ is a reflexive $\s$-ideal of $R\otimes_kK$.
\end{lemma}
\begin{proof}
Let $(\lambda_i)$ be a $k$-basis of $K$ and $s=\sum r_i\otimes\lambda_i \in R\otimes_k K$ with $\s(s)=0$. Then $\sum \sigma(r_i)\otimes\sigma(\lambda_i)=0$ implies $\sigma(r_i)=0$, because the family $(\s(\lambda_i))$ is $k$-linearly independent by Lemma \ref{lemma: stays linearly independent}. Since $\s$ is injective on $R$, $s=0$.
The latter claim of the lemma follows by applying the above result to $R/\ida$.
\end{proof}

\begin{prop} \label{prop:new phi constants are sigma constrained}
Let $K$ be a $\fs$-field such that $K^{\f^\infty}=K^\f$ and $\s\colon K^\f\to K^\f$ is surjective. Let $R$ be a $\f$-simple $K$-$\fs$-algebra that is a $\f$-pseudo domain and finitely $\s$-generated over $K$. Then $R^\f$ is a finitely $\s$-generated constrained $\s$-field extension of $K^\f$.
\end{prop}
\begin{proof} We set $k=K^\f$. The assumption $K^{\f^\infty}=K^\f$ means that $k$ is relatively algebraically closed in $K$. We also know that $K$ is separable over $k$ \cite[Cor.~1.4.16, p.~16]{Wibmer:thesis}. Thus, $K$ is a regular field extension of $k$.
Let $c$ be a finite tuple with coordinates in $R^\f$. Then $$K\{c\}_\s=K\otimes_kk\{c\}_\s$$ is an integral domain, because $k\{c\}_\s$ is contained in the field $R^\f$ and $K$ is regular over $k$. Moreover, $(0)\subset R$ is a finite intersection of $\s$-pseudo prime ideals of $R$ by Lemma \ref{lemma:finite intersection of sigmapseudoprime ideals}. We can thus apply Theorem \ref{theo:sigmaChevalley} to the inclusion $K\{c\}_\s\subset R$ to find  $0\ne r\in K\{c\}_\s$ and an integer $l\geq 1$ such that every $\s^d$-prime ideal $\q'$ of $K\{c\}_\s$ with $r\notin_\s\q'$ lifts to a $\s^{ld}$-prime ideal of $R$.
We may write
\[r=\lambda_1\otimes a_1+\cdots+\lambda_m\otimes a_m\in K\otimes_kk\{c\}_\s=K\{c\}_\s\]
with the $\lambda_i$'s linearly independent over $k$. Let $b\in k\{c\}_\s$ denote one of the non-zero $a_i$'s.
We will show that $k{\left\{c,1/b\right\}}_\s$ has no $\s$-pseudo prime ideals other than $(0)$. Let $\q$ be a $\s^d$-prime ideal of $k\{c\}_\s$ with $b\notin_\s\q$ (for some $d\geq 1$). We have to show that $\q = (0)$.

Since $K$ is a regular field extension of $k$, $\q':=K\otimes\q$ is a prime ideal of $K\otimes_kk\{c\}_\s$.
It follows from Lemma \ref{lemma: stays injective} that $\q'$ is a $\s^d$-prime ideal of $K\otimes_kk\{c\}_\s$.
We claim that $r\notin_\s\q'$. Suppose the contrary. Then $\s^n(r)\in\q'$ for some $n\geq 1$. By Lemma \ref{lemma: stays linearly independent}, the family $(\s^n(\lambda_i))$ is linearly independent over $k$. By considering the image of $\s^n(r)$ in
\[(K\otimes_k k\{c\}_\s)/\q'=K\otimes_k(k\{c\}_\s/\q),\]
we see that this implies $\s^n(b)\in\q$. This contradicts $b\notin_\s\q$. Therefore, $r\notin_\s\q'$.
By the construction of $r$, this implies the existence of a $\s^{ld}$-prime ideal $\q''$ of $R$ with $$\q''\cap K\{c\}_\s=\q'.$$ In particular, $\q''\supset \q R$. But, since the elements of $\q$ are $\f$-constants, $\q R$ is a $\f$-ideal. Since $R$ is $\f$-simple, we must have $\q R=(0)$. So, also $\q=(0)$ as desired.

It remains to see that $R^\f$ is finitely generated as a $\s$-field extension of $k=K^\f$.
Let $\q$ be a minimal prime ideal of $R$. Then there exists $d\geq 1$ such that $\q$ is $\f^d$-prime and $\s^d$-prime (Lemma \ref{lemma:finite intersection of sigmapseudoprime ideals}).
Since $R$ is finitely generated as $K$-$\s$-algebra, we see that $R/\q$ is finitely generated as $K$-$\s^d$-algebra. So, $\Q(R/\q)$ is finitely generated as $\s^d$-field extension of $K$. As $k=K^{\f^d}$ by assumption, it follows from Lemma \ref{lemma: stays simple after constant extension} that $K\otimes_kR^\f$ is $\f^d$-simple. Therefore, the canonical map $$K\otimes_kR^\f=K\cdot R^\f\to \Q(R/\q)$$ is injective, and we can think of $KR^\f=\Q(K\cdot R^\f)$ as a $\s^d$-subfield of $\Q(R/\q)$. By \cite[Thm.~4.4.1, p.~292]{Levin:difference}, every intermediate difference field of a finitely generated difference field extension is finitely generated. Therefore, 
$KR^\f$ is finitely generated as a $\s^d$-field extension of $K$. A fortiori, $KR^\f$ is finitely generated as $\s$-field extension of $K$.
We can, therefore, find  $a_1,\ldots,a_m\in R^\f$ such that
$$KR^\f=K\langle a_1,\ldots, a_m\rangle_\s.$$ So, $$\Q{\left(K\otimes_{k}R^\f\right)}=\Q{\left(K\otimes_{k}k\langle a_1,\ldots,a_m\rangle_\s\right)}.$$ As $K\otimes_{k}R^\f$ and $K\otimes_{k}K^\f\langle a_1,\ldots,a_m\rangle_\s$ are $\f$-simple (Lemma \ref{lemma: stays simple after constant extension}), it follows from Lemma \ref{lemma: constants of fsimple ring} that
$$R^\f=\Q{\left(K\otimes_{k}R^\f\right)}^\f=\Q{\left(K\otimes_{k}k\langle a_1,\ldots,a_m\rangle_\s\right)}^\f=k\langle a_1,\ldots,a_m\rangle_\s.\let\qedsymbol\openbox\qedhere
$$
\end{proof}

\begin{cor} \label{cor: no new constants}
Let $K$ be a $\fs$-field and $R$ a $\s$-\PV\ ring over $K$ with $K^\f$ being a $\s$-closed $\s$-field.
Then $R^\f=K^\f$.
\end{cor}
\begin{proof}
Since a $\s$-closed $\s$-field is algebraically closed and inversive, the hypotheses of Proposition \ref{prop:new phi constants are sigma constrained} are met, and it follows that $R^\f$ is a constrained $\s$-field extension of $k$. By \cite[Ex.~2.8, p.~1388]{Wibmer:Chevalley}, a $\s$-closed $\s$-field cannot have a proper constrained $\s$-field extension.
\end{proof}

\begin{lemma} \label{lemma: exists sigmapseudoprimeideal}
Let $K$ be a $\fs$-field and $R$, $R'$ $\s$-\PV\ rings over $K$. Then there exists a $\s$-pseudo prime ideal in $R\otimes_K R'$.
\end{lemma}
\begin{proof}
We begin the proof with a general observation on $\f$-pseudo $\s$-fields. Let $L$ be a $\f$-pseudo $\s$-field. Then $L$ need not be a $\s$-pseudo field. However, if we write $$L=e_1\cdot L\oplus\cdots\oplus e_t\cdot L$$ as after Definition \ref{defi: fpseudofield}, then $\s$-permutes the $e_i$'s and it follows that $L$ is a finite direct sum (or product) of $\s$-pseudo fields (cf. Lemma \ref{lemma:finite intersection of sigmapseudoprime ideals}.) In other words, there are idempotent elements $f_1,\ldots,f_m\in L$ such that
\[L=f_1\cdot L\oplus\cdots\oplus f_m\cdot  L,\]
with the $f_i\cdot L$'s $\s$-pseudo fields.
Set $L=\Q(R)$ and $L':=\Q(R')$. It suffices to show that there exists a $\s$-pseudo prime ideal in $L\otimes_K L'$, because a $\s$-pseudo prime ideal of $L\otimes_K L'$ contracts to a $\s$-pseudo prime ideal of $R\otimes_K R'$.
As above, we can write $$L=f_1\cdot L\oplus\cdots\oplus f_m\cdot  L\quad \text{and}\quad L'=f'_1\cdot L'\oplus\cdots\oplus f'_{m'}\cdot  L'$$ with the $f_i\cdot L$'s and $f'_j\cdot L'$'s $\s$-pseudo fields. Then
\[L\otimes_K L'=\bigoplus\nolimits_{i,j}f_i\cdot L\otimes_K f_j'\cdot L'.\]
Note that the $f_i\cdot L$'s are finitely $\s$-generated as $\s$-pseudo field extensions of $K$. Indeed, if $Y\in\Gl_n(L)$ is a suitable fundamental solution matrix, then $f_i\cdot L=K\langle f_i\cdot Y\rangle_\s$.
Since $f_1\cdot L$ is finitely $\s$-generated over $K$, it follows from \cite[Thm.~1.2, p.~1375]{Wibmer:Chevalley} that there exists a $\s$-pseudo prime ideal $\p$ in $f_1\cdot L\otimes_K f_1'\cdot L'$. Then
\[\widetilde{\p}:=\p\bigoplus_{i,j \atop (i,j)\neq (1,1)}f_i\cdot L\otimes_K f_j'\cdot L'\]
is a $\s$-pseudo prime ideal of $L\otimes_K L'$.
\end{proof}

Finally we are prepared to prove our main uniqueness theorem:

\begin{theo}[Uniqueness of $\sigma$-\PV\ rings] \label{theo: uniqueness of sigmaPVrings}
Let $K$ be a $\fs$-field such that $K^\f=K^{\f^\infty}$ and $\s\colon K^\f\to K^\f$ is surjective. Let $R_1$ and $R_2$ be two $\s$-\PV\ rings over $K$ for the same equation $\f(y)=Ay$, $A\in\Gl_n(K)$. Then there exists a finitely $\s$-generated constrained $\s$-pseudo field extension $k'$ of $k:=K^\f$ containing $k_1:=R_1^\f$ and $k_2:=R_2^\f$ and an isomorphism of $K\otimes_k k'$-$\fs$-algebras between $R_1\otimes_{k_1}k'$ and $R_2\otimes_{k_2}k'$.
\end{theo}
\begin{proof}
We know from Proposition \ref{prop:new phi constants are sigma constrained} that $k_1$ and $k_2$ are finitely $\s$-generated constrained $\s$-field extensions of $k$. Let $Y_1\in\Gl_n(R_1)$ and $Y_2\in\Gl_n(R_2)$ be fundamental solution matrices for $\f(y)=Ay$. Set $$Z=(Y_1\otimes1)^{-1}(1\otimes Y_2)\in\Gl_n(R_1\otimes_K R_2).$$ As noted in Remark \ref{rem: fundamental solution matrix}, we have
$$Z\in\Gl_n{\left((R_1\otimes_K R_2)^\f\right)}.$$ Since $1\otimes Y_2=(Y_1\otimes 1)\cdot Z$, the entries of $1\otimes Y_2$ lie in $$R_1\cdot S_Z,\quad S_Z := k_1{\left\{Z,1/\det(Z)\right\}}_\s\subset R_1\otimes_K R_2.$$ Using Lemma \ref{lemma: linearly disjoint}, it follows that
\[R_1\otimes_K R_2=R_1\cdot S_Z=R_1\otimes_{k_1} S_Z.\]
Our next goal is to find a $k_1$-$\s$-morphism $\psi\colon  S_Z\to k'$ for some finitely $\s$-generated constrained $\s$-pseudo field extension $k'$ of $k_1$.
We know from Lemma \ref{lemma: exists sigmapseudoprimeideal} that there exists a $\s$-pseudo prime ideal in $R_1\otimes_K R_2$. This $\s$-pseudo prime ideal contracts to a $\s$-pseudo prime ideal of $S_Z$. We can thus apply \cite[Prop.~2.12, p.~1390]{Wibmer:Chevalley} to find a maximal element $\p$ in the set of all $\s$-pseudo prime ideals of $S_Z$ ordered by inclusion. By \cite[Prop.~2.9, p.~1389]{Wibmer:Chevalley}, the residue $\s$-pseudo field
$$k':=\Q(S_Z/\p)$$ is a constrained $\s$-pseudo field extension of $k$. Moreover, we have a natural $k_1$-$\s$-morphism
$\psi\colon  S_Z\to k'$. Then
\[\varphi\colon R_2\to R_1\otimes_K R_2=R_1\otimes_{k_1}S_Z\xrightarrow{\id\otimes\psi}R_1\otimes_{k_1} k'\]
is a morphism of $K$-$\fs$-algebras. Since $(R_1\otimes_{k_1} k')^\f=k'$, this yields an embedding of $k_2$ into $k'$, and we can extend $\varphi$ to a $K\otimes_k k'$-$\fs$-morphism
$$\varphi\colon R_2\otimes_{k_2} k'\to R_1\otimes_{k_1}k'.$$
As $\varphi(Y_2)$ and $Y_1$ are fundamental solution matrices in $R_1\otimes_{k_1}k'$ for $\f(y)=Ay$, there exists $C\in\Gl_n(k')$ such that $$Y_1=\varphi(Y_2)C=\varphi(Y_2C).$$ Since $R_1$ is $\s$-generated by $Y_1$, this shows that $\varphi$ is surjective.
Now $R_2\otimes_{k_2}k'$ need not be $\f$-simple. However, by Lemma \ref{lemma: correspondence for constants}, every $\f$-ideal of $R_2\otimes_{k_2}k'$ is of the form $R_2\otimes_{k_2}\idb$ for some ideal $\idb$ of $k'$. Since the kernel of $\varphi$ is a $\f$-ideal, this implies that $\varphi$ is injective.
\end{proof}

\begin{lemma} \label{lemma:torsor}
Let $K$ be a \fpsf and $R$ a $\s$-Picard-Vessiot ring over $K$ with $R^\f=K^\f=:k$. Then
$$R\otimes_K R=R\otimes_k(R\otimes_K R)^\f.$$
\end{lemma}
\begin{proof}
This follows as in the beginning of the proof of Theorem \ref{theo: uniqueness of sigmaPVrings} (with $R_1=R_2=R$).
\end{proof}

\begin{cor}[Uniqueness of $\sigma$-\PV\ extensions] \label{cor:uniqueness of sigmaPVextension}
Let $K$ be a $\fs$-field and let $L_1,L_2$ be two $\s$-\PV\ extensions for the same equation $\f(y)=Ay$, $A\in\Gl_n(K)$. Assume that $K^\f$ is $\s$-closed. Then there exists an integer $l\geq 1$ and an isomorphism of $K$-$\phi\s^l$-algebras between $L_1$ and $L_2$.
\end{cor}
\begin{proof}
Let $R_1\subset L_1$ and $R_2\subset L_2$ denote the corresponding $\s$-\PV\ rings. As usual, we set $k:=K^\f$. We have $R_1^\f=k$ and $R_2^\f=k$. By Theorem~\ref{theo: uniqueness of sigmaPVrings}, there exists a finitely $\s$-generated constrained $\s$-pseudo field extension $k'$ of $k$ and an isomorphism
\[\varphi\colon R_1\otimes_k k'\to R_2\otimes_k k'\]
of $K\otimes_k k'$-$\fs$-algebras. But, by \cite[Ex.~2.8, p.~1388]{Wibmer:Chevalley}, every finitely $\s$-generated constrained $\s$-pseudo field extension of a $\s$-closed $\s$-field is trivial. This means that there exists an integer $l\geq 1$ such that $k'$ is of the form $k'=k\oplus\cdots\oplus k$ with $\s$ given by
\[\s(a_1\oplus\cdots\oplus a_l)=\s(a_l)\oplus\s(a_1)\oplus \cdots\oplus\s(a_{l-1}).\]
Let $\ida$ be a maximal ideal of $k'$. Then $\ida$ is a $\s^l$-ideal with $k'/\ida=k$ as $\s^l$-rings. For $i=1,2$, the ideal $R_i\otimes_k\ida$ is a $\f\s^l$-ideal of $R_i\otimes_k k'$, and
$\varphi$ is mapping $R_1\otimes_k\ida$ bijectively onto $R_2\otimes_k\ida$. Passing to the quotient, we obtain an isomorphism
\[\overline{\varphi}\colon (R_1\otimes_k k')/(R_1\otimes_k\ida)\to (R_2\otimes_k k')/(R_2\otimes_k\ida)\]
of $\fs^l$-rings.
But
\[R_i\to R_i\otimes_k k'\to  (R_i\otimes_k k')/(R_i\otimes_k\ida)=R_i\otimes_k(k'/\ida)=R_i\]
identifies $R_i$ with $(R_i\otimes_k k')/(R_i\otimes_k\ida)$ as $\f\s^l$-ring. So, we have constructed a $K$-$\f\s^l$-isomorphism between $R_1$ and $R_2$. Finally, this isomorphism extends to the total quotient rings, that is, to the $\s$-\PV\ extensions. 
\end{proof}

Let $K$ be a $\fs$-field and $A\in\Gl_n(K)$. Even if, in all generality, a $\s$-\PV\ extension for $\f(y)=Ay$ need not be unique, the following remark shows that, in some situations, it is possible to make a more or less canonical choice. For example, if $K=k(z)$ as in \S\ref{subsec: special existence}, then the $\s$-\PV\ ring for $\f(y)=Ay$ inside $\seq_k$ is unique (as a subring of $\seq_k$).

\begin{rem}
Let $K$ be a $\fs$-field and $A\in\Gl_n(K)$. Let $S$ be a $K$-$\fs$-algebra with $S^\f=K^\f$. If there exists a $\s$-\PV\ ring $R$ for $\f(y)=Ay$ in $S$, then $R$ is unique in the sense that any other $\s$-\PV\ ring for $\f(y)=Ay$ in $S$ equals $R$.
\end{rem}
\begin{proof}
Let $R'$ be another $\s$-\PV\ ring for $\f(y)=Ay$ inside $S$. As $R$ and $R'$ are $\s$-generated by appropriate fundamental solution matrices, it follows from Remark~\ref{rem: fundamental solution matrix} and the fact that $S^\f\subset K$ that $R'=R$.
\end{proof}

\subsection{$\s$-Galois group and Galois correspondence}\label{sec:Galoiscorrespondence}
In this section, we will define the $\s$-Galois group of $\f(y) = Ay$ (Definition~\ref{def:sGalois}), show that it is a 
$\s$-algebraic group (Lemma~\ref{lem:repbyH}), establish the Galois correspondence (Theorem~\ref{thm:Galoiscorrespondence}), and finish by showing that the $\s$-dimension, introduced in~\cite{DiVizioHardouinWibmer:DifferenceGaloisofDifferential}, of the $\s$-Galois group coincides with the $\s$-dimension of a $\s$-\PV\ ring of the equation (Lemma~\ref{lem:dimequal}), which we will further use in our applications, Theorems~\ref{thm:app} and~\ref{thm:Nev}. For this, we first recall what a 
$\s$-algebraic group is using the language of $\s$-Hopf algebras (and representable functors). See the appendix of \cite{DiVizioHardouinWibmer:DifferenceGaloisofDifferential} for a brief introduction to $\s$-algebraic groups.

Throughout Sections \ref{sec:Galoiscorrespondence} and \ref{sec:isomonodromic} we will make the following assumptions. $K$ is a \fpsf and $k:=K^\f$ is its $\s$-field of $\f$-constants. $R$ is a $\s$-Picard-Vessiot ring for the linear $\f$-equation $\f(y)=Ay,\ A\in\Gl_n(K)$ with $R^\f=k$. $L=\quot(R)$ is the corresponding $\s$-Picard-Vessiot extension. The category of $k$-$\s$-algebras is denoted by $\AlgKs$.

\begin{defi} A $k$-$\s$-Hopf algebra is a Hopf algebra over $k$ in which
the comultiplication $\Delta$, antipode $S$, and counit $\varepsilon$  are $k$-$\s$-algebra homomorphisms.
\end{defi}

\begin{defi} A 
$k$-$\s$-algebraic group is a functor  $G : \AlgKs \to \Set$ represented by a $k$-$\s$-Hopf algebra $H$, which is finitely $\s$-generated over $k$. That is, for every $B \in \AlgKs$, $$
G(B) = \Hom_{\Ks}(H,B).
$$
For simplicity, we say that $H$ represents $G$.
\end{defi}

\begin{defi}\cite[Def.~A.37]{DiVizioHardouinWibmer:DifferenceGaloisofDifferential} A 
$k$-$\s$-algebraic group $G'$ is called a $k$-$\s$-subgroup of a 
$k$-$\s$-algebraic group $G$ if $G'(B)$ is a subgroup of $G(B)$ for every $k$-$\s$-algebra $B$.
\end{defi}

\begin{prop}\cite[Rem.~A.38]{DiVizioHardouinWibmer:DifferenceGaloisofDifferential}\label{prop:subgroup} For every $k$-$\s$-algebraic subgroup $G'$ of a  
$k$-$\s$-algebraic group $G$ represented by $H$, there exists a $\s$-Hopf ideal $I$ in $H$ such that
$G'$ is represented by $H/I$ and vice versa.
\end{prop}

The multiplicative $k$-$\s$-algebraic group $\Gm$ is the $k$-$\s$-algebraic group represented by $k\{x, 1/x\}_\s$ with $\Delta(x) = x\otimes x$, $S(x) = 1/x$, and $\varepsilon(x) = 1$.

\begin{prop}\cite[Lem.~A.40]{DiVizioHardouinWibmer:DifferenceGaloisofDifferential}\label{prop:Gmphi} For every $\s$-Hopf ideal $I$ of $ H := k\{x, 1/x\}_\s$ with the above Hopf algebra structure, there exists a multiplicative function $\varphi = x^{n_0}\cdot\s(x)^{n_1}\cdot\ldots\cdot\s^t(x)^{n_t} \in H$ such that $I$ contains $\varphi-1$.
\end{prop}

\begin{lemma}\label{lem:Hopf} The $k$-$\s$-algebra
$$
H := (R\otimes_K R)^\f
$$
 is a $k$-$\s$-Hopf algebra via the $\fs$-$R$-bimodule structure on $C:=R\otimes_KR$ (see \cite[(1.5,1.6)]{AmanoMasuokaTakeuchi:HopfPVtheory}):
\begin{align*}
&\Delta : C\to C\otimes_RC, &&\Delta(a\otimes b) = a\otimes 1\otimes b \in R\otimes_K R\otimes_K R \cong R\otimes_KR\otimes_RR\otimes_KR,\\
&\varepsilon : C \to R,&&\varepsilon(a\otimes b) = ab
\end{align*}
and the $K$-$\fs$-linear flip homomorphism $
\tau : C\to C$, $\tau(a\otimes b) = b\otimes a$. Moreover,
\begin{equation}\label{eq:torsor}
\mu: R\otimes_{k}H \to R\otimes_KR,\quad r\otimes h \mapsto (r\otimes 1)\cdot h
\end{equation}
is an isomorphisms of $K$-$\fs$-algebras.
\end{lemma}
\begin{proof} The proof is a modification of the proof of \cite[Prop.~1.7]{AmanoMasuokaTakeuchi:HopfPVtheory} and \cite[Prop.~3.4]{AmanoMasuoka:artiniansimple}.
We already noted in Lemma \ref{lemma:torsor} that (\ref{eq:torsor}) is an isomorphism. 
It follows that the $K$-$\fs$-algebra homomorphism
$$
\begin{CD}
R\otimes_{k}H\otimes_{k} H @>\mu\otimes\id>>R\otimes_KR\otimes_{k} H@>\id\otimes\mu>> R\otimes_KR\otimes_KR
\end{CD}
$$
is an isomorphism.
By taking $\f$-constants, we, therefore, obtain a $k$-$\s$-algebra isomorphism
\begin{equation}\label{eq:HtoA}
H\otimes_{k} H \to (R\otimes_K R\otimes_K R)^\f.
\end{equation}
To show that, given the above, $H$ becomes a $k$-$\s$-Hopf algebra, one proceeds as in the 
proof of \cite[Prop.~1.7]{AmanoMasuokaTakeuchi:HopfPVtheory}.\end{proof}

\begin{defi}\label{def:sGalois} Let $R$ and $L$ be as above. Then the $\s$-Galois group of $L$ over $K$ is defined as the functor
$$
\Gal^\s(L|K) : \AlgKs \to \Set,\quad B\mapsto\Gal^\s(L|K)(B) := \AutKfs\big(R\otimes_{k}B\big|K\otimes_{k} B\big),  
$$
where $\f$ acts as the identity on $B$.
\end{defi}

\begin{lemma}\label{lem:repbyH} Let $R$, $L$, and $H$ be as above. Then $G:=\Gal^\s(L|K)$ is a  
$k$-$\s$-algebraic group represented by $H$.
\end{lemma}
\begin{proof}
As in the proof of \cite[Lem.~1.9]{AmanoMasuokaTakeuchi:HopfPVtheory}, $R$ is an $H$-comodule via
$$
\theta : R \to R\otimes_{k}H,\quad r \mapsto \mu^{-1}(1\otimes r),
$$
which is a $K$-$\fs$-algebra homomorphism,
where $\mu$ is defined in~\eqref{eq:torsor}. 
For every $k$-$\s$-algebra B and  $g \in \Hom_{\Ks}(H,B)$, we have a $K$-$\fs$-algebra homomorphism
$$
\begin{CD}
\Phi_g : R\otimes_{k} B @>\theta\otimes\id_B>>R\otimes_{k}H\otimes_{k}B@>\id_R\otimes g\otimes\id_B>> R\otimes_{k}B\otimes_{k}B @>\id_R\otimes m>> R\otimes_{k}B,
\end{CD}
$$
which is an automorphism by \cite[Thm.~3.2]{Waterhouse}. Moreover, by \cite[Thm.~3.2]{Waterhouse} as well, the map $g \mapsto \Phi_g$ is a group homomorphism. For the reverse direction, let $Y\in\Gl_n(R)$ be a fundamental solution matrix of $\f(y)=Ay$ and
$Z=(Y\otimes 1)^{-1}(1\otimes Y)\in\Gl_n(R\otimes_K R)$. Then $H=k\{Z,1/\det(Z)\}_\s$ and it follows from Remark~\ref{rem: fundamental solution matrix} that, for any $$\varphi \in \AutKfs\big(R\otimes_{k}B\big|K\otimes_{k} B\big),$$ there exists $C_\varphi \in \Gl_n(B)$ such that
$\varphi(Y) = Y\cdot C_\varphi$. We define a $k$-$\s$-algebra homomorphism $H \to B$ by sending $Z$ to $C_\varphi$. 
\end{proof}

\begin{theo}\label{thm:Galoiscorrespondence} There is a one-to-one correspondence between $k$-$\s$-algebraic subgroups in $G$ and intermediate $\f$-pseudo $\s$-fields in $L|K$ given by
\begin{equation}\label{eq:Galoiscorrespondence}
M = L^{G'} := \big\{a/b \in L\:|\: \theta'(a)\cdot b=a\cdot\theta'(b),\ a,b \in R\big\}\quad\longleftrightarrow\quad G' := \Gal^\sigma(L|M),
\end{equation}
or, alternatively,
\begin{equation}\label{eq:Galoiscorrespondence2}
M = L^{G'} := \big\{x \in L\:|\: \text{for all } B \in \AlgKs,\: g \in G'(B),\ g(x\otimes 1) = x\otimes 1\big\}\quad\longleftrightarrow\quad G' := \Gal^\sigma(L|M),
\end{equation}
where $\theta' : R\to R\otimes_{k} H'$, and $H'$ represents $G'$.
\end{theo}
\begin{proof} We will  show that there is a one-to-one correspondence between the $\s$-Hopf ideals  in $H$ and intermediate $\f$-pseudo $\s$-fields in $L|K$ given by 
$$
M = \{x \in L\:|\: 1\otimes x - x\otimes 1 \in I\cdot (L\otimes_K L)\}\quad\longleftrightarrow\quad I = H\cap\ker(L\otimes_KL\to L\otimes_ML).
$$
The proof below is partly an adaptation of \cite[Prop.~2.3]{AmanoMasuokaTakeuchi:HopfPVtheory}.
It follows from \cite[Thm.~3.1.17]{Wibmer:thesis} that there is a one-to-one correspondence between the intermediate $\f$-pseudo $\s$-fields in $L|K$ and $\fs$-coideals of $L\otimes_KL$ given by
$$
M  = \{x \in L\:|\: 1\otimes x - x\otimes 1 \in J \subset L\otimes_KL\}\quad\longleftrightarrow\quad J = \ker(L\otimes_KL\to L\otimes_ML).
$$
By Lemma~\ref{lemma: correspondence for constants}, there is a one-to-one correspondence between $\fs$-ideals of $L\otimes_{k}H$ and $\s$-ideals  of $H$ given by 
\begin{equation}\label{eq:HandL}
\idb = \ida\cap H \quad\longleftrightarrow\quad \ida = L\otimes_{k}\idb.
\end{equation} 
By localizing~\eqref{eq:torsor}, we obtain $K$-$\fs$-algebra isomorphisms
\begin{equation}\label{eq:LHR}
\varphi_1 : L \otimes_{k} H \to L\otimes_K R\quad\text{and}\quad \varphi_2 : L\otimes_{k}H \to R\otimes_K L.
\end{equation}
Therefore, we have a one-to-one correspondence between $\s$-ideals of $H$ and $\fs$-ideals of $L\otimes_KL$ given by composing~\eqref{eq:LHR} and~\eqref{eq:HandL} and using the fact that the set of ideals of the localization $L\otimes_K L$ consists of the intersection of the set of ideals in the smaller localizations $L\otimes_KR$ and $R\otimes_KL$ inside the set of ideals in $R\otimes_KR$.

We will now show that, under the above correspondence and in the above notation, $(L\otimes_KL)\cdot\ida$ is a $\fs$-coideal of $L\otimes_KL$ if and only if $\idb$ is a $\s$-Hopf ideal of $H$. For this, note that, similarly to the above, we have a one-to-one correspondence between ideals in $H\otimes_{k} H$ and $\f$-ideals in $L\otimes_KL\otimes_KL$. Indeed, by  Lemma~\ref{lemma: correspondence for constants} and isomorphsims~\eqref{eq:HtoA} and~\eqref{eq:LHR}, there is a one-to-one correspondence between $\fs$-ideals of $L\otimes_KR\otimes_KR$ (as well as those in $R\otimes_KL\otimes_KR$ and $R\otimes_KR\otimes_KL$) and $\s$-ideals  of $H\otimes_{k}H$ with 
$$
H\otimes_{k}\idb\ \ \longleftrightarrow\ \ \ida = L\otimes_K\idb\quad\text{and}\quad \idb\otimes_{k}H\ \ \longleftrightarrow\ \ \ida = \idb\otimes_KL,
$$
therefore,
$$
\idb_1 := H\otimes_{k}\idb + \idb\otimes_{k}H\quad\longleftrightarrow\quad\ida_1 :=  L\otimes_K\ida + \ida\otimes_KL
$$
under the correspondence $\ida \subset L\otimes_KL\leftrightarrow \idb \subset H$ from the preceding paragraph. Therefore, $$
\Delta(\ida)\subset \ida_1,\ \varepsilon(\ida)= 0\quad\Longleftrightarrow\quad \Delta(\idb)\subset \idb_1,\ \varepsilon(\idb)= 0.$$
By \cite[Theorem~1(iv)]{Nichols}, $\idb$ is a Hopf ideal of $H$ if and only if $\idb$ is a coideal of $H$, which finishes the proof.
To show correspondence~\eqref{eq:Galoiscorrespondence}, note that, by Lemma~\ref{lem:repbyH}, $\Gal^\s(L|M)$ is represented by $$H/{H\cap\ker(L\otimes_KL\to L\otimes_ML)}.$$ Therefore, it remains to show that 
\begin{align*}&L_1 := \{x \in L\:|\: 1\otimes x - x\otimes 1 \in \ker(L\otimes_KL\to L\otimes_ML)\} = L_2 := \big\{a/b \in L\:|\: \theta'(a)\cdot b=a\cdot\theta'(b),\ a,b \in R\big\}=\\
&=  L_3 := \big\{x \in L\:|\: \text{for all } B \in \Alg_{\Ks},\: g \in G'(B),\ g(x\otimes 1) = x\otimes 1\big\}.
\end{align*}
For every $x=a/b \in L_2$, $B \in \Alg_{\Ks}$, and $g \in G'(B)$, we have
$$
g(a/b\otimes 1)= (\theta'(a)\cdot b\otimes 1)(g)/(\theta'(b)\cdot b\otimes 1)(g) = (\theta'(b)\cdot a\otimes 1)(g)/(\theta'(b)\cdot b\otimes 1)(g) = a/b\otimes 1.
$$
Hence, $x \in L_3$. Now, for all $x=a/b \in L_3$, we have $\theta'(a)\cdot b=a\cdot\theta'(b)$ by taking $B := H$ and $g := \id_H$. Therefore, $L_2 = L_3$. For $L_1=L_3$, see the proof of \cite[Lem.~3.1.11]{Wibmer:thesis}.
\end{proof}

For $\s$-dimension, see \cite[\S A.7]{DiVizioHardouinWibmer:DifferenceGaloisofDifferential}. Let $K$ be a $\fs$-field and $R$, $L$, and $H$ be as above.

\begin{lemma}\label{lem:dimequal}
We have 
$$
\sdim_KR = \sdim_{k}H.
$$
\end{lemma}
\begin{proof} Let $Y\in \Gl_n(R)$ be
a fundamental solution matrix of $\f(y) = Ay$ and $Z=(Y\otimes 1)^{-1}(1\otimes Y)\in\Gl_n(R\otimes_K R)$. Then $$R = K\{Y,1/\det(Y)\}_\s\quad \text{and}\quad H = k\{Z,1/\det(Z)\}_\s.$$
The claim now follows from \cite[Def.~A.25]{DiVizioHardouinWibmer:DifferenceGaloisofDifferential}, Lemma~\ref{lemma:higherorder}, and \cite[Thm.~1.13]{SingerPut:difference}.
\end{proof}

\subsection{Isomonodromic difference equations}\label{sec:isomonodromic}
In this section, we develop a $\s$-Galois treatment for isomonodromic difference equations. In particular, in Theorem~\ref{thm:isomcrit}, not assuming that the field $k=K^\f$ is difference closed, we give a criterion, which says that $\f(y)=Ay$ is isomonodromic if and only if the matrices in its $\s$-Galois group all satisfy an equation of a special form~\eqref{eq:isomGalois}. This result is a difference analogue of  the corresponding results for isomonodromic differential equations, \cite[Prop.~3.9]{PhyllisMichael} and \cite[Thm.~6.6]{GO}. We further illustrate this by considering a $q$-hypergeometric equation in Example~\ref{ex:qhyp}.
\begin{defi} The system $\f(y)=Ay$ is called {\it isomonodromic} if there exists $B \in \Gl_n(K)$ such that 
\begin{equation}\label{eq:CI}
\f(B) = \s(A)BA^{-1}.
\end{equation}
\end{defi}

\begin{theo}\label{thm:isomcrit} The equation $\f(y)=Ay$ is isomonodromic if and only if there exists $D \in \Gl_n(k)$ such that the
following equation is in the defining ideal of the $\s$-Galois group $G$:
\begin{equation}\label{eq:isomGalois}
\s(x_{ij})=D^{-1}(x_{ij})D.
\end{equation}
Moreover, if~\eqref{eq:isomGalois} is in the defining ideal of $G$, then there exists a finitely generated $\s$-field extension $F$ of $k$ and $C \in \Gl_n(F)$
such that 
\begin{equation}\label{eq:constant}
\s\big(C^{-1}(x_{ij})C\big)=C^{-1}(x_{ij})C
\end{equation} is in the defining ideal of $G$, that is, $G$ is conjugate over $F$ to a group of matrices with $\s$-constant entries.
\end{theo}
\begin{proof} 
Let $Y \in \Gl_n(R)$ be a fundamental solution matrix. 
Let there exist $B \in \Gl_n(K)$ such that~\eqref{eq:CI} is satisfied. We have
$$
\f\big(\s(Y)^{-1}BY\big) = \s(\f(Y))^{-1}\f(B)\f(Y)=\s(AY)^{-1}\s(A)BA^{-1}AY=\s(Y)^{-1}BY.
$$
Therefore, there exists $D \in \Gl_n(k)$ such that $
\s(Y)=BYD$. 
For every $k$-$\s$-algebra $S$ and  $g \in G(S)$,  let $C_g \in \Gl_n(S)$ be such that $g(Y) = YC_g$. Then, on the one hand, $$g(\s(Y)) = g(BYD) = BYC_gD.$$ On the other hand, $$g(\s(Y))= \s(g(Y))=\s(YC_g) = \s(Y)\s(C_g)=BYD\s(C_g).$$ Therefore, for all $g \in G$, we have
$$\s(C_g) = D^{-1}C_gD,$$
showing~\eqref{eq:isomGalois}. To show~\eqref{eq:constant}, let $F$ be a $\s$-field generated over $k$ by
the entries of an invertible matrix $C$ satisfying $\s(C)=D^{-1}C$. Then, 
$$\s\big(C^{-1}C_gC\big)=\s(C)^{-1}\s(C_g)\s(C)=C^{-1}DD^{-1}C_gDD^{-1}C=C^{-1}C_gC.$$

Suppose now that, for all $k$-$\s$-algebras $S$ and  $g \in G(S)$, we have
$\s(C_g) = D^{-1}C_gD$, where $C_g := Y^{-1}g(Y)$. Let $B := \s(Y)D^{-1}Y^{-1}$. Then, for all $g \in G$,
$$
g(B) = \s(YC_g)D^{-1}(YC_g)^{-1} = \s(Y)D^{-1}C_gDD^{-1}C_g^{-1}Y^{-1} = B.
$$
By Theorem~\ref{thm:Galoiscorrespondence}, $B \in \Gl_n(K)$. We, moreover, have
$$
\f(B) = \f(\s(Y))D^{-1}\f(Y)^{-1}=\s(AY)D^{-1}(AY)^{-1}=\s(A)\s(Y)D^{-1}Y^{-1}A^{-1}=\s(A)BA^{-1},
$$
showing~\eqref{eq:CI}.
\end{proof}

\begin{ex}\label{ex:qhyp}
Consider a $q$-hypergeometric equation 
\begin{equation}\label{eq:qdiff}
y\big(q^2x\big) -\frac{2ax-2}{a^2x -1}y(qx) + \frac{x-1}{a^2x -1}y(x) = 0.
\end{equation}
It is shown in~\cite{Roques} that, over $\C(x)$, if $a \notin q^\ZZ$, then, if $a^2 \notin q^\ZZ$, then the
difference Galois group of~\eqref{eq:qdiff} is $\Gl_2(\C)$, otherwise it is $\Sl_2(\C)$.
Equation~\eqref{eq:qdiff} has been also studied from the differential-parametric viewpoint in~\cite[Ex.~3.14]{HardouinSinger:DifferentialGaloisTheoryofLinearDifferenceEquations}.
Let now $\CC$ be any field such that~\eqref{eq:qdiff} has a $\s$-PV extension over $\CC(x,a)$, with $a$ being transcendental over $\CC(x,a)$, $\f$ and $\s$ acting as $\id$ on $\CC$, and $$\f(x) = qx,\ \f(a) = a,\ \s(x) = x,\ \s(a)=qa.$$
The existence can be shown as in Proposition~\ref{prop:special existence}.
A calculation in {\sc Maple}, similar to the one given in~\cite{CharlotteMichaelMaple}, but using the procedure {\tt RationalSolution} in the {\tt QDifferenceEquations} package, shows that~\eqref{eq:qdiff}, once transformed into the matrix form, is isomonodromic over $\CC(x,a)$ with
$$
B =
\begin{pmatrix}\cfrac{1}{a^2x-1}&-\cfrac{2a}{(a+1)(a^2x-1)}\\ \cfrac{2a(x-1)}{(a+1)(a^2x-1)(a^2qx-1)}&\cfrac{3a-1+(a^3-3a^2)x}{(a+1)(a^2x-1)(a^2qx-1)}
\end{pmatrix}.$$
Therefore,~\eqref{eq:CI} is in the defining ideal of the $\s$-PV group $G$ of~\eqref{eq:qdiff} by Theorem~\ref{thm:isomcrit}. 
It follows from \cite[Cor.~3.3.2.1]{Andre} and~\cite[Thm.~10]{Roques} that the (non-$\s$-parametric) PV group of~\eqref{eq:qdiff} over $\CC(x,a)$ is $\Gl_2$.  Similarly to \cite[Prop.~6.21]{HardouinSinger:DifferentialGaloisTheoryofLinearDifferenceEquations}, it follows from Theorem~\ref{thm:Galoiscorrespondence} that $G$ is Zariski dense in $\Gl_2$. It follows from Theorem~\ref{thm:isomcrit} that, $G$ is conjugate to $\Gl_2(\CC)$ over a (proper, as {\tt RationalSolution} shows) finitely-generated $\s$-field extension of $\CC(a)$, where $\Gl_2(\CC)$ is defined by $\Gl_2(\CC)(B) = \{g\in\Gl_2(B)\:|\: \s(g)=g\}$ for every $\CC(a)$-$\s$-algebra $B$.
\end{ex}

\section{Applications and examples}\label{sec:examples}
In this section, we will illustrate how our Galois theory can be used to study difference and differential algebraic properties of functions. We start by showing a general $\s$-independence criterion in Theorem~\ref{thm:app} (see also \cite[Thm.~4.1]{OvchinnikovTrushinetal:GaloisTheory_of_difference_equations_with_periodic_parameters}). In \S\ref{thm:app}, we show a $\s$-independence criterion over the field of meromorphic function with Nevanlinna growth order less than one (Theorem~\ref{thm:Nev}). For this, we need some preparatory work, Lemmas~\ref{lem:Nev} and~\ref{lem:fromM1}, which are
interesting on their own as they generalize a natural modification of a classical result in complex analysis \cite{BankKaufman}. We then show how to apply our results in practice in Theorem~\ref{thm:6}, which is followed by illustrative examples in \S\ref{sec:finalexamples}.

\subsection{General result}
\begin{theo}\label{thm:app}
Let $F$ be a $\fs$-field containing the field $\C(z)$ with
\begin{equation}\label{eq:actiononz}
\f(z) = a_1z+a_2,\quad\s(z)=b_1z+b_2,\quad a_1,a_2,b_1,b_2 \in \C,\quad a_1b_1\ne 0,\quad \f\s=\s\f,\quad \f^n\ne \id,\ n\in\NN,
\end{equation} and 
$\kk:=F^{\f}$. 
 Let $0\ne f\in F$ and $0\ne a\in {\C(z)}$ be such
that $f$ is 
a 
solution of
\begin{equation}\label{eq:exdiffeq}
\f(y) = ay.
\end{equation}
Then $f$ is $\s$-algebraically dependent over the field $\kk(z)$ if
and only if
\begin{equation}\label{eq:aqbb}
\varphi(a)=\f(b)/b
\end{equation}
for some $0\ne b \in \C(z)$ and
$1\ne\varphi(x)=x^{n_0}{\s(x)}^{n_1}\cdot\ldots\cdot
{\s^{t-1}(x)}^{n_{t-1}}$.
\end{theo}
\begin{proof}
 If~\eqref{eq:aqbb} holds, then
$$
\f(\varphi(f)/b)=
\varphi(\f(f))/\f(b)=\varphi(af)/\f(b)=\varphi(a)\varphi(f)/\f(b)=\varphi(f)/b.
$$
Therefore,
$
\varphi(f)/b=c\in F^\f=\kk$.
Thus,
$
\varphi(f)=c\cdot b\in \kk(z)$,
which gives a $\s$-algebraic dependence for $f$ over $\kk(z)$.

Assume now that $f$ is $\s$-algebraically dependent over
$\kk(z)$.  
Let $L$ be the smallest
$\fs$-subfield in $F$ containing $\kk(z)$ and $f$. 
Since $\kk \subset L^\f \subset F^\f = \kk$,  
the $\fs$-field $L$ is a $\s$-\PV\ extension over
$\kk(z)$ for equation~\eqref{eq:exdiffeq}. 
It follows from Lemma~\ref{lem:dimequal} and Proposition~\ref{prop:Gmphi} that $f$ is
$\s$-algebraically dependent over $\kk(z)$ if and only if the
$\s$-Galois group $G$ of $L|K$ is a
proper $\s$-algebraic subgroup of $\Gm$. Then, by
Proposition~\ref{prop:Gmphi}, there exists a multiplicative
$
\varphi\in \kk\{x,1/x\}_{\s}
$
 such that the ideal of $G$ contains the equation
$
\varphi(x) = 1$. Therefore, for every $\kk$-$\s$-algebra  $B$ and $g\in G(B)$, we have
$$
g(\varphi(f))= \varphi(g(f))=\varphi(c_g\cdot
f) = \varphi(c_g)\cdot \varphi(f)=1\cdot \varphi(
f)= \varphi(f).
$$
Hence, by Theorem~\ref{thm:Galoiscorrespondence}, we have
$
b:= \varphi(f) \in \kk(z)$. Since $f\ne 0$ 
and $\varphi$ is
multiplicative, $\varphi(f)\ne 0$.  
Therefore,
\begin{align}\label{eq:dep}
\varphi(a) = \varphi(\f(f)/f)=
\f(\varphi(f))/\varphi(f)= \f(b)/b.
\end{align}

We will show now that $b$ can be chosen from $\C(z)$
satisfying~\eqref{eq:aqbb}. For this, first note that $z$ is transcendental over $\kk$. Indeed, for all $n\in\NN$ and $a_0,\ldots,a_n \in\kk$,
$$
a_nz^n+\ldots+a_1z+a_0=0
$$ implies that, for all $q\in\NN$,
$$
a_n(\f^q(z))^n+\ldots+a_1(\f^q(z))+a_0=0.
$$
This implies that there exists $r\in\NN$ such that $z=\f^r(z)$, which contradicts~\eqref{eq:actiononz}.
Now, we have the
equalities $a=\bar a/c$ and $b=\bar b /d$, where
$\bar a, c\in \C[z]$ and $\bar b, d\in \kk[z]$. Consider the
coefficients of $\bar b$ and $d$ with respect to $z$ as new indeterminates. 
 Equation~\eqref{eq:dep} is equivalent to
$$
\varphi(\bar a/c)=\f{\left(\bar b/d\right)}\big/{\left(\bar b/d\right)}.
$$
So, we have
\begin{equation}\label{eq:phiabcd}
\varphi(\bar a)\cdot\f(d)\cdot\bar b-\varphi(c)\cdot\f(\bar b)\cdot d=0.
\end{equation}
The left-hand side of equation~\eqref{eq:phiabcd} is a polynomial in
$z$. Hence, equation~\eqref{eq:dep} can be considered as a system of polynomial equations
 given by the equalities for
all coefficients.
 Since the field $\C$ is  algebraically closed, existence of
$\bar b$ and $d$ with coefficients in $\kk$ implies
existence of  
$\bar b$ and $d$ with coefficients in
$\C$.
\end{proof}

\subsection{Meromorphic functions and Nevanlinna property}\label{sec:Nev}

Let $M$ be the $\fs$-field of meromorphic functions on the plane with
$$
\f(f)(z) := f(z+1),\quad \s(f)(z) := f(z+a_\s),\quad f \in M,\quad z, a_\s \in \C.
$$
Also, let $\kk := M^\f$, which is the field of $1$-periodic meromorphic functions.
For  $f\in M$, the standard  Nevanlinna characteristics  $m(r,f)$, $N(r,f)$, and $T(r,f)$ were introduced in \cite[pp.~6, 12]{Nevanlinna1974} (see also \cite{GoldOst,BankKaufman,ChiangFeng}).
Let
\begin{equation}\label{eq:Nevlimit}
M_{<1} := {\big\{g \in M\:\big|\: T(r,g) = o(r),\ r\to+\infty\big\}},
\end{equation}
which is a $\fs$-field as well \cite[8.~Prop.]{BankKaufman}. Note that 
\begin{equation}\label{eq:CzK}
\C(z) \subsetneq M_{<1}.
\end{equation}

The proof of the following result, which we need to prove Theorem~\ref{thm:Nev}, was suggested by D.~Drasin and S.~Merenkov, to whom the authors are highly grateful, as
a modification of \cite[7.~Lem.~(c)]{BankKaufman}.
\begin{lemma}\label{lem:Nev}
Let $f\in M$ and there exist $R\in \C(z)$ such that, for all $z \in \C$, 
\begin{equation}\label{eq:fzplusone}
f(z+1)=R(z)\cdot f(z).
\end{equation}
If $f\in M_{<1}$, then $f \in \C(z)$.
\end{lemma}
\begin{proof}
Let $L >0$ be a real number such that all finite poles and zeroes of $R$ lie in $$D(L) := \{c \in\C\:|\: |c| < L\}.$$
Similarly to the proof of \cite[7.~Lem.~(c)]{BankKaufman}, one shows that~\eqref{eq:fzplusone} and~\eqref{eq:Nevlimit} imply that
all finite poles and zeroes of $f$ lie in $D(L)$. This implies that there exists a rational function $h$ such that $g:=hf$ is an entire function with no zeroes. Since $M_{<1}$ is a field and $h \in M_{<1}$, we have $g=hf\in M_{<1}$. Hence, it follows from \cite[Lem.~I.6.2]{GoldOst} that $g$ is constant. Therefore, $f=g/h$ is rational.
\end{proof}

\begin{cor} \label{cor: constants of M1}

 We have
$$
\kk\cap M_{<1} = \C.
$$
\end{cor}

We will need one more complex-analytic result (which has an algebraic proof) to prove Theorem~\ref{thm:Nev} as well.

\begin{lemma}\label{lem:fromM1}
Let $a\in\C(z)\smallsetminus\{0\}$. Assume that there exists a non-zero $b\in \kk\, M_{<1}$ such that $\f(b)=ab$. Then 
\begin{enumerate}
\item\label{it:341} there also exists a non-zero $b'\in M_{<1}$ with $\f(b')=ab'$,
\item\label{it:342} $b \in \kk(z)$.
\end{enumerate}
\end{lemma}
\begin{proof}
We know from Corollary \ref{cor: constants of M1} that $M_{<1}^\f=\C$, and it follows from Lemma \ref{lemma: linearly disjoint} that $M_{<1}$ is linearly disjoint from $\kk$ over $\C$. Hence, 
\begin{equation}\label{eq:quot}
\kk\, M_{<1}=\quot\big(M_{<1}\otimes_{\C}\kk\big).
\end{equation}
Moreover, $M_{<1}\otimes_{\C}\kk$ is $\f$-simple by Lemma \ref{lemma: stays simple after constant extension}.
We will first show that $b$ must lie in $M_{<1}\otimes_{\C}\kk$. Set
\[\ida=\left\{f\in M_{<1}\otimes_{\C}\kk\:|\:f\cdot b\in M_{<1}\otimes_{\C}\kk\right\}.\]
It follows from~\eqref{eq:quot} that $\ida$ is a non-zero ideal of $M_{<1}\otimes_{\C}\kk$. For all $f\in\ida$, we have $\f(fb)\in M_{<1}\otimes_{\C}\kk$ and, therefore,
\[\f(fb)=\f(f)\cdot ab\in M_{<1}\otimes_{\C}\kk.\]
Since $a\in\C(z)\subset M_{<1}$, this implies $$\f(f)\cdot b\in M_{<1}\otimes_{\C}\kk,$$ that is, $\f(f)\in\ida$. So, $\ida$ is a $\f$-ideal. Since $M_{<1}\otimes_{\C}\kk$ is $\f$-simple, we must have $1\in\ida$. So, $$b\in M_{<1}\otimes_{\C}\kk.$$
Choose a $\C$-basis $(c_i)$ of $\kk$ and write $b=\sum_i b_i\otimes c_i$ with $b_i\in M_{<1}$.
Then
\[\sum\nolimits_i\f(b_i)\otimes c_i=\f(b)=ab=\sum\nolimits_i ab_i\otimes c_i.\]
Hence, for all $i$, we have $\f(b_i)=ab_i$. By Lemma~\ref{lem:Nev}, we conclude that,
for all $i$, $b_i \in \C(z)$, which implies that $b \in \kk(z)$, showing~\eqref{it:342}. Moreover, since $b \ne 0$, there exists $i$ such that $b_i\ne 0$, showing~\eqref{it:341}.
\end{proof}

\begin{theo}\label{thm:Nev}
Let $f\in M$ and $0\ne a\in \C(z)$ be such
that $f$ is a non-zero solution of
\begin{equation}\label{eq:exdiffeq2}
\f(y) = ay.
\end{equation}
Then $f$ is $\s$-algebraically dependent over  $M_{<1}$ if
and only if
\begin{equation}\label{eq:aqbb2}
\varphi(a)=\f(b)/b
\end{equation}
for some $0\ne b \in \C(z)$ and
$1\ne\varphi(x)=x^{n_0}{\s(x)}^{n_1}\cdot\ldots\cdot
{\s^{t-1}(x)}^{n_{t-1}}$.
\end{theo}
\begin{proof} The converse follows as in Theorem~\ref{thm:app}, noting \eqref{eq:CzK} and Corollary \ref{cor: constants of M1}. Let now $f$ be $\s$-algebraically dependent over $M_{<1}$. As in the proof of Theorem~\ref{thm:app}, we will show that there exists $b \in M_{<1}$ and multiplicative $\varphi$ such that 
$$
\varphi(a) = \f(b)/b.
$$
Lemma~\ref{lem:Nev} implies that $b \in \C(z)$. To do the above, let $L$ be the smallest
$\fs$-subfield in $M$ containing $\kk$, $M_{<1}$, and $f$. 
Since $\kk \subset L^\f \subset M^\f = \kk$,   
the $\fs$-field $L$ is a $\s$-\PV\ extension over
$\kk\,M_{<1}$ for equation~\eqref{eq:exdiffeq}. 
It follows from Lemma~\ref{lem:dimequal} and Proposition~\ref{prop:Gmphi} that $f$ is
$\s$-algebraically dependent over $\kk\, M_{<1}$ if and only the
$\s$-Galois group $G$ of equation~\eqref{eq:exdiffeq} is a
proper $\s$-algebraic subgroup of $\Gm$. Then, by
Proposition~\ref{prop:Gmphi}, there exists a multiplicative
$
\varphi\in \kk\{x,1/x\}_{\s}
$
 such that the ideal of $G$ contains the equation
$
\varphi(x) = 1$. Therefore, for every $\kk$-$\s$-algebra  $B$ and $g\in G(B)$, we have
$$
g(\varphi(f))= \varphi(g(f))=\varphi(c_g\cdot
f) = \varphi(c_g)\cdot \varphi(f)=1\cdot \varphi(
f)= \varphi(f).
$$
Hence, by Theorem~\ref{thm:Galoiscorrespondence}, we have
$
b:= \varphi(f) \in \kk\, M_{<1}$. Since $f\ne 0$ 
and $\varphi$ is
multiplicative, $\varphi(f)\ne 0$.  
Therefore,
\begin{align*}
\varphi(a) = \varphi(\f(f)/f)=
\f(\varphi(f))/\varphi(f)= \f(b)/b.
\end{align*}
By Lemma~\ref{lem:fromM1}, there exists $b' \in M_{<1}$ such that $\varphi(a) = \phi(b')/b'$, which finishes the proof.
\end{proof}

\subsection{How to use the above results in practice}
Let $a\in \C(z)^\times$ and $w_0,\:z_0\in \C^\times$ and $\f$ and $\s$ act on $\C(z)$ as follows:
$$
\f(f)(z)=f(z+w_0)\quad\text{and}\quad \s(f)(z)=f(z+z_0),\quad f \in \C(z).
$$
Then, for some $N \ge 0$, $a$ can
be represented as follows
$$
a = \lambda\cdot
\prod_{k=0}^{t-1}\prod_{d=-N-1}^{N}\prod_{i=1}^{R}\left( z - k\cdot z_0-d\cdot w_0- r_i \right)^{s_{k,d,i} },
$$
where $\lambda, r_i\in \mathbb C$ and the $r_i$'s are distinct in
$\mathbb C\big/w_0\cdot{\mathbb Z}+z_0\cdot\mathbb Z$.  For all $i$ and $k$, $1\le i \le R$, $0\le k\le t-1$, let
\begin{equation}\label{eq:defaik}
a_{i,k} = \sum_{d=-N-1}^{N} s_{k,d,i}.
\end{equation}
The following result combined with Theorems~\ref{thm:app} and~\ref{thm:Nev} provides  a complete characterization of all equations~\eqref{eq:exdiffeq} whose solutions are $\s$-algebraically independent.

\begin{theo}\label{thm:6}
Let $a\in \C(z)$
be as above and $z_0/w_0 \notin \mathbb Q$. Then
\begin{enumerate}
\item If  $\lambda$ is a root of unity, then
there exist  $b\in \C(z)$ and a
multiplicative function 
$$
\varphi(x)=x^{n_0}\cdot \left(\s
(x)\right)^{n_1}\cdot\ldots\cdot{\left(\s^A(x)\right)}^{n_A}\neq 1
$$
such that $\varphi(a)=\f(b)/b$ if and only if, for all $i$, $1\le i \le R$, $$a_{i,0}=\ldots=a_{i,t-1}=0.$$
\item If  $\lambda$ is not a root of unity, then
there exist  $b\in \mathbb C(z)$ and a
multiplicative function
$$
\varphi(x)=x^{n_0}\cdot \left(\s(x)\right)^{n_1}\cdot\ldots\cdot {\left(\s^A (x)\right)}^{n_A}\neq 1
$$
such that $\varphi(a)=\f(b)/b$ if and only if, for all $i$, $1\le i \le R$, $$a_{i,0}=\ldots=a_{i,t-1}=0\quad\text{and}\quad t\ge 2.$$\end{enumerate}
\end{theo}
\begin{proof}
We will write $\varphi$ and $b$ with undetermined coefficients and
exponents. Suppose that
$$
b=\mu\cdot 
\prod_{k=-B}^{B}\prod_{d=-N}^{N}\prod_{i=1}^{R}\left( z - k\cdot z_0 - 
d\cdot w_0- r_i \right)^{l_{k, d, i}}
\quad
\text{and}\quad
\varphi(x)=x^{n_0}\cdot \left(\s(x)\right)^{n_1}\cdot\ldots\cdot
{\left(\s^{A}(x)\right)}^{n_{A}}
$$
are such that  $\varphi(a)=\f(b)/b$ and $A,B \ge 0$. Let us
calculate the right and left-hand sides of this equality. We
see that
$$
\f(b)=\mu\cdot 
\prod_{k=-B}^{B}\prod_{d=-N}^{N}\prod_{i=1}^{R}\left( z - k\cdot z_0 - 
(d-1)\cdot w_0- r_i \right)^{l_{k, d, i}}.
$$
Hence,
$$
\begin{aligned}
&\frac{\f(b)}{b}= 
\prod_{k=-B}^{B} \prod_{d=-N-1}^{N-1}\prod_{i=1}^{R} \left(  z - k\cdot z_0 - 
d\cdot w_0- r_i  \right)^{l_{k, d+1, i}}\cdot
\prod_{k=-B}^{B} \prod_{d=-N}^{N}\prod_{i=1}^{R} \left(  z - k\cdot z_0 - 
d\cdot w_0- r_i  \right)^{-l_{k, d, i}}=\\
&=
\prod_{k=-B}^{B}\prod_{i=1}^{R} \Bigl[\left(  z - k\cdot z_0 + 
(N+1)\cdot w_0- r_i \right)^{l_{k, -N, i}}\prod_{d=-N}^{N-1} \left( z - k\cdot z_0 - 
d\cdot w_0- r_i \right)^{l_{k,
d+1, i}-l_{k, d, i}}
\left(  z - k\cdot z_0 - 
N\cdot w_0- r_i
\right)^{-l_{k, N, i}}\Bigr].
\end{aligned}
$$
Now, we calculate the left-hand side. We see that, for all $r \ge 0$, 
$$
\begin{aligned}
\s^r(a)^{n_r} &= \lambda^{n_r}\cdot
\prod_{k=0}^{t-1}\prod_{d=-N-1}^{N}\prod_{i=1}^{R}\left(  z - (k-r)\cdot z_0 - 
d\cdot w_0- r_i \right)^{n_rs_{k, d, i} }=\lambda^{n_r}\cdot
\prod_{k=-r}^{t-1-r}\prod_{d=-N-1}^{N}\prod_{i=1}^{R}\left(  z - k\cdot z_0 - 
d\cdot w_0- r_i \right)^{n_r s_{r+k, d, i}}.
\end{aligned}
$$
Hence,
$$
\begin{aligned}
\varphi(a)&=\lambda^{\sum_{r=0}^{A} n_r}\cdot
\prod_{k=-A}^{t-1}\prod_{d=-N-1}^{N}\prod_{i=1}^{R}\left(  z - k\cdot z_0 - 
d\cdot w_0- r_i\right)^{\sum_{0\le r\le A\atop{0\le r+k\le t-1}}n_r s_{r+k , d, i}}.
\end{aligned}
$$
Now, the equation $\varphi(a)=\f(b)/b$ gives $A=B = t-1$ and the following
system of linear equations
$$
\left\{
\begin{aligned}
&\left\{
\begin{aligned}
\sum\nolimits_{0\le r,\: r+k\le t-1} s_{r+k, -N-1, i}\cdot n_r &= l_{k, -N, i}\\
\sum\nolimits_{0\le r,\: r+k\le t-1} s_{r+k, d, i}\cdot n_r &= l_{k, d+1, i } - l_{k, d, i},\quad \quad -N\leqslant d\leqslant N-1,\quad 1 \le i \le R,\quad 1-t\le k\le t-1 \\
\sum\nolimits_{0\le r,\: r+k\le t-1} s_{r+k, N, i} \cdot n_r &= -l_{k, N, i}\\
\end{aligned}\right.\\
&\lambda^{\sum_{r=0}^{t-1}n_r} = 1
\end{aligned}\right.
$$
The first subsystem, for all $i$ and $k$, $1\le i \le R$, can be rewritten as follows:
$$
\left\{
\begin{aligned}
\begin{pmatrix}
s_{0, -N-1, i}\\
s_{0, -N, i}\\
\vdots\\
s_{0, N, i}\\
\end{pmatrix}
\begin{pmatrix}
n_{t-1}
\end{pmatrix}
&=
\begin{pmatrix}
l_{k, -N, i}\\
l_{k, -N +1, i} - l_{k, -N, i}\\
\vdots\\
-l_{k, N, i}\\
\end{pmatrix},\quad k=1-t.\\
\begin{pmatrix}
s_{0, -N-1, i}&s_{1, -N-1, i}\\
s_{0, -N, i}&s_{1, -N, i}\\
\vdots&\vdots\\
s_{0, N, i}&s_{1, N, i}\\
\end{pmatrix}
\begin{pmatrix}
n_{t-2}\\
n_{t-1}
\end{pmatrix}
&=
\begin{pmatrix}
l_{k, -N, i}\\
l_{k, -N +1, i} - l_{k, -N, i}\\
\vdots\\
-l_{k, N, i}\\
\end{pmatrix},\quad k=2-t.\\
&\vdots\\
\begin{pmatrix}
s_{0, -N-1, i}&s_{1, -N-1, i}&\ldots&s_{t-1, -N-1, i}\\
s_{0, -N, i}&s_{1, -N, i}&\ldots&s_{t-1, -N, i}\\
\vdots&\vdots&\ddots&\vdots\\
s_{0, N, i}&s_{1, N, i}&\ldots&s_{t-1, N, i}\\
\end{pmatrix}
\begin{pmatrix}
n_0\\
n_1\\
\vdots\\
n_{t-1}
\end{pmatrix}
&=
\begin{pmatrix}
l_{k, -N, i}\\
l_{k, -N +1, i} - l_{k, -N, i}\\
\vdots\\
-l_{k, N, i}\\
\end{pmatrix},\quad k=0.\\
&\vdots\\
\begin{pmatrix}
s_{t-1, -N-1, i}\\
s_{t-1, -N, i}\\
\vdots\\
s_{t-1, N, i}\\
\end{pmatrix}
\begin{pmatrix}
n_0
\end{pmatrix}
&=
\begin{pmatrix}
l_{k, -N, i}\\
l_{k, -N +1, i} - l_{k, -N, i}\\
\vdots\\
-l_{k, N, i}\\
\end{pmatrix},\quad k=t-1.
\end{aligned}\right.
$$
Each subsystem has a solution in $l_{k, d, i}$ if and only if the sum
of all equations is zero. Thus, we can replace this system with the
following:
$$
\left\{
\begin{aligned}
n_{t-1}\cdot\sum_{d=-N-1}^{N} s_{0,d,i}&=0\\
n_{t-2}\cdot \sum_{d=-N-1}^{N} s_{0,d,i}+n_{t-1}\cdot \sum_{d=-N-1}^{N}s_{1,d,i}&=0\\
&\vdots\\
n_0\cdot \sum_{d=-N-1}^{N} s_{0,d,i}+n_1\cdot \sum_{d=-N-1}^{N}s_{1,d,i}+\ldots+n_{t-1}\cdot\sum_{d=-N-1}^{N}s_{t-1, d, i}&=0\\
&\vdots\\
n_0\cdot\sum_{d=-N-1}^{N} s_{t-1,d,i}&=0
\end{aligned}
\right.
$$
Using~\eqref{eq:defaik}, we obtain the following
system:
\begin{equation}\label{eq:firstcase}
\begin{pmatrix}
0&0&\ldots&0&a_{i,0}\\
0&0&\ldots&a_{i,0}&a_{i,1}\\
\vdots&\vdots&\vdots&\vdots&\vdots\\
a_{i,0}&a_{i,1}&\ldots&a_{i,t-2}&a_{i,t-1}\\
\vdots&\vdots&\vdots&\vdots&\vdots\\
a_{i,t-2}&a_{i,t-1}&\ldots&0&0\\
a_{i,t-1}&0&\ldots&0&0\\
\end{pmatrix}
\begin{pmatrix}
n_0\\
n_1\\
\vdots\\
n_{t-1}\\
\end{pmatrix}
=\begin{pmatrix}
0\\\vdots\\ 0
\end{pmatrix}.
\end{equation}
Thus, for some integers $\gamma_{k,d,i,j}$, we have:
$$
\left\{
\begin{aligned}
&\begin{pmatrix}
0&0&\ldots&0&a_{i,0}\\
0&0&\ldots&a_{i,0}&a_{i,1}\\
\vdots&\vdots&\vdots&\vdots&\vdots\\
a_{i,0}&a_{i,1}&\ldots&a_{i,t-2}&a_{i,t-1}\\
\vdots&\vdots&\vdots&\vdots&\vdots\\
a_{i,t-2}&a_{i,t-1}&\ldots&0&0\\
a_{i,t-1}&0&\ldots&0&0\\
\end{pmatrix}
\begin{pmatrix}
n_0\\
n_1\\
\vdots\\
n_{t-1}\\
\end{pmatrix}
=\begin{pmatrix}
0\\\vdots\\ 0
\end{pmatrix}\\
&\lambda^{\sum_{r=0}^{t-1}n_r} = 1\\
&l_{k, d, i} = \sum_{r=0}^{t-1} \gamma_{k, d, i, r}\cdot n_r
\end{aligned}\right.
$$

Consider the {\bf first case}:  $\lambda$ is a root of unity. Then, for some $u \in \mathbb Z\setminus \{0\}$, we have
$\lambda^u=1$. In this situation, if $n_r, l_{k,d,i}$ is a
solution of all equations except for the second one, then
$$
u\cdot n_r,\ \ u\cdot l_{k,d,i}
$$
is a solution of the whole system.
Therefore, in this case, the existence of $\varphi$ and $b$ is
equivalent  to~\eqref{eq:firstcase}  
having a nontrivial common solution.

Consider the {\bf second case}:  $\lambda$ is not a root of unity. Then the second equation
gives
$\sum_{r=0}^{t-1} n_r=0$. Thus, in this case, we need to show the existence
of a nontrivial solution of the system
\begin{equation}\label{eq:secondcase}
\begin{aligned}
&\begin{pmatrix}
0&0&\ldots&0&a_{i,0}\\
0&0&\ldots&a_{i,0}&a_{i,1}\\
\vdots&\vdots&\vdots&\vdots&\vdots\\
a_{i,0}&a_{i,1}&\ldots&a_{i,t-2}&a_{i,t-1}\\
\vdots&\vdots&\vdots&\vdots&\vdots\\
a_{i,t-2}&a_{i,t-1}&\ldots&0&0\\
a_{i,t-1}&0&\ldots&0&0\\
1&1&\ldots&1&1
\end{pmatrix}
\begin{pmatrix}
n_0\\
n_1\\
\vdots\\
n_{t-1}\\
\end{pmatrix}
=\begin{pmatrix}
0\\\vdots\\ 0
\end{pmatrix}.
\end{aligned}
\end{equation}
Since all the coefficients in~\eqref{eq:firstcase} and~\eqref{eq:secondcase}
are integers,  there is a nontrivial solution with integral
coefficients if and only if there is a nontrivial solution with
complex coefficients. 

In the {\bf first case}, the rank is less than $t$ if and only if
$$
a_{i,0}=\ldots=a_{i,t-1}=0.
$$
In the {\bf second case}, the rank is less than $t$ if and only if
$$
a_{i,0}=\ldots=a_{i,t-1}=0\quad \text{and}\quad t \ge 2.\let\qedsymbol\openbox\qedhere
$$
\end{proof}
\subsection{Examples}\label{sec:finalexamples}
We will now illustrate Theorems~\ref{thm:app},~\ref{thm:Nev}, and~\ref{thm:6}.
\begin{ex}\label{ex:Gmmaf} The gamma function $\Gamma$ satisfying
$$
\Gamma(z+1)=z\cdot\Gamma(z)
$$
does not satisfy any polynomial difference equation over  $M_{<1}$ (see~\eqref{eq:Nevlimit}) for any shift by $z_0\notin \mathbb Q$ as, in the notation of Theorem~\ref{thm:6}, $N=0$, $t=1$, $R =1$, and
$1=s_{0,0,1} = a_{1,0}\ne 0$. A differential algebraic independence statement over $M_{<1}$ for $\Gamma$ was shown in \cite{BankKaufman} using analytic techniques. Also, \cite[Thm.~1]{ChiangFeng} gives difference algebraic independence of the Riemann zeta function $\zeta$ over $M_{<1}$. Note the following relation between $\zeta$ and $\Gamma$:
$$
\zeta(1-s) = 2^{1-s}\cdot\pi^{-s}\cdot\cos(\pi\cdot s/2)\cdot \Gamma(s)\cdot\zeta(s).
$$
\end{ex}

\begin{ex}\label{ex:power}
For $f\in K := \C(z,\alpha)$, let
$$
\f(f)(z,\alpha) = f(z,\alpha+1)\quad\text{and}\quad\s(f)(z,\alpha)=f(z+z_0,\alpha).
$$
Let $F$ be a $\fs$-field over $K$ that contains a non-zero solution of
$$
\f(y)=z\cdot y,
$$
which we denote by $z^\alpha$. Let $\varphi$ be as in the statement of Theorem~\ref{thm:app}. If $z^\alpha$ were $\s$-algebraic dependent over $\C(z,\alpha)$, then, by the proof of Theorem~\ref{thm:app}, there would exist $0 \ne b \in F^\f$ (note that $\C(z)\subset F^\f$ in our case) such that
$$
1  =  \f(b)/b=\varphi(z).$$
Since $\s$ is a shift, $\varphi = 1$, which is a contradition. This proves the difference algebraic independence of $z^\alpha$ over $\C(z,\alpha)$ with respect to shifts of $z$ (see \cite{FGS2005,BGKL2008} for a related statement, in which $\alpha$ takes values in $\QQ$).
\end{ex}
 
\begin{ex}\label{ex:differential}
Let $K$ be a field. Consider $\seq_K$ as a $\s$-ring with $\s$ acting as the shift. Let $L$ be a $\s$-subfield of $\seq_K$. Consider $\seq_L $as a $\fs$-ring with $\f$ acting as the shift and $\s$ acting coordinatewise. Let F be a $\fs$-subfield of $\seq_L$ and $\{S(m,\alpha)\}\in F$ satisfy a first-order $\f$-difference equation
$$S(m,\alpha+1)=f(m,\alpha)\cdot S(m,\alpha),\quad \{f(m,\alpha)\}\in M,$$
where $M$ is a $\fs$-subfield of F, which contains $L$. Then it follows from the proof of Theorem~\ref{thm:app}, \cite[Lem.~A.40]{DiVizioHardouinWibmer:DifferenceGaloisofDifferential}, and \cite[Prop.~1.1]{EisenbudSturmfels} that, if $\{S(m,\alpha)\}$ satisfies a linear
$\s$-difference equation, then there exists $\{b(m,\alpha)\} \in M$ and $n\Ge 1$ such that,
for all $m$ and $\alpha$,
\begin{equation}\label{eq:Smnalpha}
S(m+n,\alpha)=b(m,\alpha)\cdot S(m,\alpha).
\end{equation}
In particular, we can take $M$ to be the image of $L(z)$ in $\seq_L$, as in \S\ref{subsec: special existence}. Let
$$
y(x,\alpha) := \sum\nolimits_{m\Ge 0} S(m,\alpha)\cdot x^m
$$
By~\cite[Thm.~1.5]{Stanley1980}, the function $y(x,\alpha)$ satisfies a linear differential equation in $x$ if and only if $S(m,\alpha)$ satisfies a homogeneous linear difference equation in $m$  (see also \cite[App.~B.4]{FSbook} and the reference given there).
Suppose it is known that $S(m,\alpha)$ satisfies a first-order homogeneous linear difference equation with respect to $\alpha$, and one wants to know whether $y(x,\alpha)$ satisfies a linear differential equation in $x$. The above method helps find difference equations in $m$ if they are hard to find otherwise, as such equations are all of the form~\eqref{eq:Smnalpha}. Just to illustrate the process (but not the difficulty), consider the Bessel functions of the first kind, which are given by
$$
J_\alpha(x) = \sum\nolimits_{m\Ge0} \frac{(-1)^m}{m!\cdot\Gamma(m+\alpha+1)}(x/2)^{2m+\alpha},
$$
where $\alpha$ is an integer. It is a solution of the following differential equation:
\begin{equation}\label{eq:Bessel}
x^2y''+xy'+{\left(x^2-\alpha^2\right)}\cdot y = 0,
\end{equation}
where $'$ stands for $\frac{\mathrm{d}}{\mathrm{d}x}$. Let 
$$
S(m,\alpha) = \frac{(-1)^m}{m!\cdot \Gamma(m+\alpha+1)}\quad\text{and}\quad I_\alpha(x) = \sum\nolimits_{m\Ge0}S(m,\alpha)x^m.
$$
Then
$
J_\alpha(x) = (x/2)^\alpha\cdot I_\alpha{\left(x^2/4\right)}$.
We have:
$$
(m+\alpha+1)\cdot S(m,\alpha+1)=S(m,\alpha).
$$
Moreover, we have:
$$
(m+1)(m+\alpha+1)\cdot S(m+1,\alpha)+S(m,\alpha)=0.
$$
Therefore, by a calculation using the {\tt Gfun} package in {\sc Maple} \cite{SalvyGfun}, $I_\alpha(x)$ satisfies the second-order  linear differential equation  
\begin{equation}\label{eq:Ia}
xy''+y'+y=0
\end{equation}
(see also the proof of~\cite[Thm.~1.5]{Stanley1980}). One now obtains~\eqref{eq:Bessel} by substituting the expression of $I_\alpha$ in terms of $J_\alpha$ into~\eqref{eq:Ia}. 
\end{ex}
\bibliographystyle{spmpsci}
\bibliography{bibdata.bib}

\end{document}